\documentclass[UTF-8,reqno]{amsart}
\usepackage{enumerate}
\usepackage{amssymb,url,color,  tikz, booktabs}
\usepackage{graphics,graphicx}

\usepackage{mathrsfs}
\usetikzlibrary{arrows, shapes.geometric}
\usetikzlibrary{decorations.pathreplacing}
\tikzstyle{arrow} = [ultrathick,>=stealth]
\tikzstyle{block} = [rectangle, minimum width=3cm, minimum height=1cm, align=flush center, draw=black, thick]

\usepackage{changes}

\makeatletter
\newcommand\RSloop{\@ifnextchar\bgroup\RSloopa\RSloopb}
\makeatother
\newcommand\RSloopa[1]{\bgroup\RSloop#1\relax\egroup\RSloop}
\newcommand\RSloopb[1]%
  {\ifx\relax#1%
   \else
     \ifcsname RS:#1\endcsname
       \csname RS:#1\endcsname
     \else
       \GenericError{(RS)}{RS Error: operator #1 undefined}{}{}%
     \fi
   \expandafter\RSloop
   \fi
  }
\newcommand\X{0}
\newcommand\RS[1]%
  {\begin{tikzpicture}
     [every node/.style=
       {circle,draw,fill,minimum size=1.5pt,inner sep=0pt,outer sep=0pt},
      line cap=round
     ]
   \coordinate(\X) at (0,0);
   \RSloop{#1}\relax
   \end{tikzpicture}
  }
\newcommand\RSdef[1]{\expandafter\def\csname RS:#1\endcsname}
\newlength\RSu
\RSdef{i}{\draw (\X) -- +(90:\RSu) node{};}
\RSdef{l}{\draw (\X) -- +(110:\RSu) node{};}
\RSdef{r}{\draw (\X) -- +(70:\RSu) node{};}
\RSdef{I}{\draw (\X) -- +(90:\RSu) coordinate(\X I);\edef\X{\X I}}
\RSdef{L}{\draw (\X) -- +(120:\RSu) coordinate(\X L);\edef\X{\X L}}
\RSdef{R}{\draw (\X) -- +(60:\RSu) coordinate(\X R);\edef\X{\X R}}

\usepackage{color}
\usepackage[colorlinks=true]{hyperref}
\hypersetup{
    linkcolor=blue,          
    citecolor=red,        
    filecolor=blue,      
    urlcolor=cyan
}

\numberwithin{equation}{section}
\def\theequation{\arabic{section}.\arabic{equation}}
\newcommand{\be}{\begin{eqnarray}}
\newcommand{\ee}{\end{eqnarray}}
 \newtheorem{theorem}{Theorem}[section]
\newtheorem{lemma}[theorem]{Lemma}

\newtheorem{definition}[theorem]{Definition}

\newtheorem{corollary}[theorem]{Corollary}

\allowdisplaybreaks

\definecolor{darkergreen}{rgb}{0.0, 0.5, 0.0}

\def\E{\mathbb E}
\def\C{\mathcal C}
\def\t{\partial_t}
\def\T{\mathbb T^2}
\def\d{\mathrm{d}}
\def\Z{\mathbb Z^2}
\def\ge{\geqslant}
\def\geq{\geqslant}
\def\le{\leqslant}
\def\leq{\leqslant}
\begin{document}

\title[Strong convergence rates]{Strong convergence rates for full-discrete approximations of the stochastic Allen-Cahn equations on 2D torus}

\author{Ting Ma}
\address[T. Ma]{College of Mathematics and Physics, Chengdu University of Technology, Chengdu 610059,  China}
\email{matingting2008@yeah.net}

\author{Lifei Wang}
\address[L. Wang]{ School of Mathematical Sciences, Hebei Normal University, Shijiazhuang 050024, China}
\email{flywit1986@163.com}

\author{Huanyu Yang*}
\address[H. Yang]{College of Mathematics and Statistics, Chongqing University, Chongqing 401331, China; Key Laboratory of Nonlinear Analysis and its Applications (Chongqing University), Ministry of Education}
\email{yanghuanyu1992@outlook.com}

\thanks{
We would like to thank Professor Rongchan Zhu for helpful discussion. T. M. is grateful to the financial supports of the NSFC (No. 12101429). *H.Yang is the Corresponding author
}

\begin{abstract}
	
In this paper we construct  space-time full discretizations    of   stochastic Allen-Cahn equations driven by space-time white noise on  2D torus.
The approximations are implemented by  tamed exponential Euler discretization in time and spectral  Galerkin method in space.  We finally
obtain the convergence rates with the spatial order of  $\alpha-\delta$  and the temporal order of ${\alpha}/{6}-\delta$  in   $\mathcal C^{-\alpha}$ for $\alpha\in(0,1/3)$ and $\delta>0$ arbitrarily small.

\end{abstract}

\subjclass[2010]{60H15; 82C28}
\keywords{stochastic Allen-Cahn equations; convergence rates; Galerkin method; tamed exponential Euler discretizations}

\date{}

\maketitle

\tableofcontents

\section{Introduction}

Consider the stochastic Allen-Cahn equation on $\T$  given by
\begin{equation}
\label{initial-Eq}
\left\{
 \begin{aligned}
\t X&= \Delta X-X+:F(X):+\xi\text{~~in~} (0,\infty)\times\T,\\
X(0)&=X_0 \text{~~on~} \T.
 \end{aligned}
 \right.
\end{equation}
Here $\mathbb T^2$ is a torus of size $1$ in $\mathbb R^2$, $\xi$ is  space-time white noise on $(0,\infty)\times\T$ (see Definition \ref{WHITE-NOISE}), $\Delta$ is the Laplacian with periodic boundary conditions on $L^2(\T)$. $F(v)=\sum_{j=0}^3a_jv^{j}$ with $a_3<0$.
For space dimension $d=2$, equation \eqref{initial-Eq} is ill-posed. The space-time white noise becomes so singular,
that   the nonlinearity $F(X)$  requires a renormalization. These kinds of singular stochastic partial differential equations
(SPDEs) have received a lot of attention recently
(see e.g. \cite{DPD03,Hai14,GIP13}). Formally, $:F(X):=\sum_{j=0}^3a_jX^{:j:}$, where $X^{:k:}$ stands for  the $k$-th Wick power of $X$ ( see  Section \ref{set-Wick powers} for its definition).
\subsection{Background and motivation}\label{moti}
The present work aims to propose a   space-time full-discrete scheme to numerically solve    \eqref{initial-Eq} on $\T$ driven by  space-time white noise. Furthermore, the investigation delves into the analysis of strong convergence rates concerning both temporal and spatial approximations.
In recent decades, considerable research efforts have been dedicated to developing numerical methodologies for evolutionary stochastic partial differential equations (SPDEs). For a comprehensive overview, refer to \cite{BH19,BJ16,CHL17,FKLL16,G98,G99,GSS16, HJLP19,HJ14,HJS17,Y05,ZZ2015}, as well as the supplementary references therein.

In the case of SPDEs with superlinearly growing nonlinearities, such as   stochastic Allen-Cahn equations, traditional methods like the classical exponential Euler method and the linear-implicit Euler method exhibit a lack of convergence in temporal approximation. For a detailed analysis, refer to \cite{BHJKL19}.
The development and examination of approximation algorithms tailored to address SPDEs with superlinearly growing nonlinearities represent a vibrant area of research. Kovacs, Larsson, and Lindgren, in their work \cite{KLL15}, established temporal convergence rates through an Euler-type split-step approach. Additionally, the authors in \cite{BG20,BJJ19} explored convergence rates by introducing various splitting time discretization schemes.
Hutzenthaler, Jentzen, and Salimova, in \cite{HJS17}, derived robust convergence rates in both temporal and spatial domains by devising a nonlinearity-truncated approximation scheme. Similar numerical approximation strategies have been proposed in \cite{BGJK2017,JP15}, along with other pertinent references, for handling SPDEs with superlinearly growing nonlinearities. Recently, Wang introduced a novel tamed exponential Euler time discretization method based on spectral Galerkin projection in \cite{W20}.

A significant body of literature addresses the numerical analysis of the stochastic Allen-Cahn equations. For equations involving trace class noise, references such as \cite{BJJ19,CM19,KLL15} are relevant. For equations with multiplicative noise, works like \cite{KW23,MP18,QZX23} can be consulted. In the case of equations with less regular space-time white noise on one-dimensional space, references like \cite{BGJK2017,BG20,BJJ19,CM19,W20} provide valuable insights. Yan’s work in \cite{Y05} established convergence rates of space-time approximations for linear SPDEs driven by space-time white noise in $d$-dimensional space, where $d$ ranges from $1$ to $3$.
However, there is a scarcity of references on numerical analysis for SPDEs with superlinearly growing nonlinearities driven by space-time white noise in high-dimensional spaces. Existing studies such as \cite{TW2018,TW20182,ZZ2015,MW2010} have focused on spatial convergence of the stochastic Allen-Cahn equations (specifically the dynamical $\Phi^4_d$
  model for $d>1$) without delving into convergence rates. In our prior work \cite{MZ2019}, we obtained the spatial convergence rate of Galerkin approximations for \eqref{initial-Eq}.
 To the best of our knowledge, there is a gap in the literature regarding space-time approximation schemes for SPDEs with superlinearly growing nonlinearities driven by space-time white noise in high-dimensional spaces. Our objective is to propose and analyze numerical methods tailored for addressing these equations.

\subsection{Main methods and results}
\label{Statement of the main results}
Using Da Prato-Debussche's  trick, $X$ is said to solve the equation \eqref{initial-Eq} if
\begin{equation}\label{sum}
 X=Y+\bar Z,
\end{equation}
where $\bar Z$ is a solution to  the linear version of \eqref{initial-Eq} with the initial value $X_0$, i.e.
\begin{equation}
\label{ini1-Eq}
\left\{
 \begin{aligned}
\t \bar Z&= A \bar Z+\xi \text{~~in~} (0,\infty)\times\T,\\
\bar Z(0)&=X_0 \text{~~on~} \T,
 \end{aligned}
 \right.
\end{equation}
with $A=\Delta-I$,   and $Y$ solves
\begin{equation}
\label{ini2-Eq}
\left\{
 \begin{aligned}
\t Y&=A Y+\Psi(Y,\underline{ \bar Z}) \text{~~in~} (0,\infty)\times\T,\\
Y(0)&=0\text{~~on~} \T,
 \end{aligned}
 \right.
\end{equation}
with $\underline{z}:=(z,z^{:2:},z^{:3:})$ and
\begin{equation}\label{wick-YZ}
\Psi(y,\underline{z}):= \sum^3_{j=0} a_j \sum ^j_{k=0} \binom{j}{k} y^k {z}^{:j-k:}.
 \end{equation}
It turns out that $Y$  takes values in a Besov
space of   positive regularity. Hence, the nonlinear terms in \eqref{wick-YZ} can be well-defined  through multiplicative inequalities in these spaces.

This paper presents a space-time approximation scheme for solving the solution to \eqref{initial-Eq} using a spectral Galerkin method and a nonlinearity-tamed exponential Euler time discretization, as proposed in \cite{W20}. Unlike other tamed and truncated exponential Euler methods (see, for example, \cite{BGJK2017,CM19}), this full-discrete scheme is distinct in that it exclusively employs spectral approximation for the linear part $\bar Z$
  without any temporal approximation. The methodology introduced in \cite{W20} enhances the temporal convergence rate by a factor of two compared to existing approaches in the literature for the same one-dimensional problem. Therefore, we employ this scheme to numerically solve the equation \eqref{initial-Eq}. The scheme is outlined as follows.
For $N\in\mathbb N$, let $\bar Z^N$ be a solution to Galerkin approximation \eqref{ini1-Eq_N} of   \eqref{ini1-Eq}.  Furthermore, let $T>0$, for $M\in\mathbb N$   we set $\tau=T/M$ being the time stepsize and $t_k:=k\tau$, $k\in\{0,1,...,M\}$, and propose a space-time full discretization of \eqref{sum}   as  for $t\in[0,T]$
(cf. \eqref{nNMX})
\begin{equation*}\label{nnNMX}\begin{aligned}
 X_t^{N,M}=&\int^{t\vee \tau}_\tau\frac{P_N S_{t-s} \Psi(Y_{{\lfloor s\rfloor}_\tau}^{N,M},\underline{\bar Z}_{{\lfloor s\rfloor}_\tau}^{N} )}
 {1+\tau \|\Psi(Y_{{\lfloor s\rfloor}_\tau}^{N,M},\underline{\bar Z}_{{\lfloor s\rfloor}_\tau}^{N} )\| _{-\alpha}}  ds  +\bar Z^N_t=:Y^{N,M}_t+\bar Z^N_t,
       \end{aligned}
\end{equation*}
with the spectral  Galerkin operator $P_N$   defined by \eqref{pro-ope}. $S_t,t\ge0$ is the  semigroup  generated by $A$.  Here and below,  $\lfloor s\rfloor_\tau=t_k$ for any $s\in[t_k,t_{k+1})$. To overcome    the blow-up of $\bar Z_t$ for t close to 0 in a negative regularity  Besov space, we perform integration starting from a positive time  $\tau$  instead of 0 as described in \cite{W20}. Equivalently, in the approximation scheme we set the nonlinear term $Y^{N,M}\equiv0$ in the interval $[0,\tau]$ (refer to Section \ref{SEC4-11} for further details).
Below we present the main result, i.e. the convergence rates for the full-discrete approximation  of  equation \eqref{initial-Eq} (cf. Theorem \ref{MAIN-THM1111}).
\begin{theorem}
Let $X_0\in\mathcal C^{-\alpha}$ with $\alpha\in (0,1/3)$, $\gamma>1-3\alpha$ and $p\ge2$. Then for  any  $\delta>0$
\[
 \Big(\E \sup_{t\in[0, T]}
 t^{\gamma p}\|X_t-X_{t}^{N,M}\|_{-\alpha}^p\Big)^{1/p}  \lesssim    N^{\delta-\alpha}+   M^{\delta-{\alpha}/{6}}
 \]
 for uniform large $N,M\in \mathbb N$.
\end{theorem}

Utilizing the decomposition \eqref{sum}, the initial step of our investigation involves evaluating the error between the linear term $\bar Z$
  and its Galerkin approximation $\bar Z^N$. This aspect has been extensively studied in the literature, with works such as \cite{RZZ2017, TW2018, TW20182} focusing on its convergence properties and \cite{MZ2019} analyzing its convergence rate. Following this, we delve into the time regularity of $\bar Z, \bar Z^N$ to further refine our estimate. The main challenge in our analysis lies in establishing the error between the nonlinear term $Y$ and its numerical approximations $Y^{N,M}$. Leveraging the smoothing properties of the semigroup $S_t, t \ge 0$
  generated by $A$, we are able to derive uniform a priori bounds for $Y^{N,M}$ in a Besov space with positive regularity (as detailed in Theorem \ref{Test-Y-HOLD-Th}). Notably, we approach the error estimate of $Y$ and $Y^{N,M}$
  in a Besov space of positive regularity $\beta$, employing two distinct norms: $\|\cdot\|_\beta$
  and $t^\gamma\|\cdot\|_\beta$ for some $\gamma > 0$ (refer to Theorems \ref{MAIN-THM1} and \ref{MAIN-THM1-2} respectively). This choice is motivated by the fact that the latter norm facilitates achieving superior convergence rates. Consequently, we establish both space and time convergence rates for the approximating scheme \eqref{nNMX}.

\subsection{Structure of the paper} \label{struc:sec}
This paper is organized as follows.
In Section \ref{subsec-BS} we collect results related to Besov spaces.


Section \ref{SEC3} is devoted to the regularity results for stochastic linear equation \eqref{ini1-Eq} and  its spectral Galerkin approximation \eqref{ini1-Eq_N}.


In Section \ref{Pathwise-error},  we begin with the construction of  a space-time full discrete scheme \eqref{nNMX} for  \eqref{initial-Eq}. Then we give the   the   convergence rates in time and space for full-discrete approximations of   \eqref{initial-Eq}, via uniform a priori bounds of   \eqref{Tmild-2-NM}.

\medskip


Throughout  the paper  we use the notations $a \lesssim b$ if there exists a constant $c > 0$ independent of the relevant quantities such that $a \leq cb$, $a\gtrsim b$ if $b\lesssim a$, and $a\backsimeq b$ if $a \lesssim b$ as well as $b\lesssim a$. We also introduce the notation $a \lesssim_{m} b$ to emphasize that the constant we omit depends on $m$. The constants may change from line to line and we omit  unless  necessary.

\section{Besov spaces and preliminaries}\label{subsec-BS}
We first recall Besov spaces from \cite{ZZ2015,MW2010}.  For general theory we refer to  \cite{BCD2011,Tri78,Tri06}.  For $p\in[1,\infty]$, let $L^p(\mathbb T^d)$ denote the usual $p$ integrable space on $\mathbb T^d$ with its norm denoted by $\|\cdot\|_{L^p}$.  The space of Schwartz functions on $\mathbb T^d$ is denoted by $\mathcal S (\mathbb T^d)$ and its dual, the space of tempered distributions is denoted by $\mathcal S'(\mathbb T^d)$. The space of real valued infinitely differentiable functions is denoted by $C^{\infty}(\mathbb T^d)$. For any function $f$ on $\mathbb T^d$, let ${\rm{supp}} (f)$ denote its support.

Consider the orthonormal basis $\{e_m\}_{m\in\mathbb Z^d}$ of trigonometric functions on $\mathbb T^d$
\begin{equation}\label{basis}
  e_m(x):=e^{\iota2\pi m \cdot x},  x\in\mathbb T^d,
\end{equation}
we write $L^2(\mathbb T^d)$ with its inner product $ \langle \cdot,\cdot\rangle$ given by
 \[
    \langle f,g\rangle=\int_{\mathbb T^d}f(x)\bar g(x)\d x,~~f,g\in L^2(\mathbb T^d).
 \]
Then for any $f\in L^2(\mathbb T^d)$, we denote by $\mathcal F f$ or $ \hat{f}$ its Fourier transform

\[
   \hat{f}(m)=\langle f, e_m\rangle=\int_{\mathbb T^d} e^{-\iota2\pi m \cdot x}f(x)\d x,~~m\in\mathbb Z^d.
\]
For  $N\in\mathbb N$, we define $P_N$ the projection operators from $L^2(\mathbb T^d)$ onto the space spanned by $ \{e_m,~|m|\leq N\}$ with $m=(m_1,\ldots,m_d)\in\mathbb Z^d$, $|m|:=\sqrt{m_1^2+\cdots+m^2_d}$. It means that
\begin{equation}\label{pro-ope}
P_N f=\sum_{m:|m|\leq N}\langle f, e_m\rangle ~e_m,~f\in L^2(\mathbb T^d).
\end{equation}

For $\zeta\in\mathbb R^d$ and $r > 0$ we denote by $B(\zeta, r)$ the ball of radius $r$ centered at $\zeta$ and let the annulus $\mathcal A:=B(0,\frac 8 3)\setminus B(0,\frac 3 4)$. According to \cite[Proposition 2.10]{BCD2011}, there exist nonnegative radial functions $\chi, \theta\in\mathcal D(\mathbb R^d)$, the space of real valued infinitely differentiable functions of compact support on $\mathbb R^d$, satisfying

(i) ${\rm {supp}}(\chi)\subset B(0,\frac 4 3),~{\rm {supp}}(\theta)\subset\mathcal A$;

(ii) $\chi(z)+\sum_{j\geq 0}\theta(z/{2^j})=1$ for all $z\in\mathbb R^d$;

(iii) ${\rm {supp}}(\chi)\cap {\rm {supp}}(\theta(\cdot/{2^j}))=\emptyset$ for $j\geq1$ and ${\rm {supp}}(\theta(\cdot/{2^i}))\cap {\rm {supp}}(\theta(\cdot/{2^j}))=\emptyset$ for $|i-j|>1$.

$(\chi,\theta)$ is called a dyadic partition of unity. The above decomposition can be applied to distributions on the torus (see \cite{S85,SW71}).  Let
\begin{equation}\label{TLPBs}
   \chi_{-1}:=\chi,~~\chi_{j}:=\theta(\cdot/{2^j}),~j\geq0.
\end{equation}
It is easy to see that
\begin{equation}\label{subset}
{\rm {supp}}(\chi_{j})\subset \mathcal A_{2^j}:=2^j \mathcal A,~ j\geq0;~ \mathcal A_{2^{-1}}\subset B(0,\frac 4 3).
\end{equation}
 For  $f\in C^\infty(\mathbb T^d)$, the $j$-Littlewood-Paley block is  defined as
\begin{equation}\label{LPBs}
\Delta_{j}f(x)=\sum_{m\in\mathbb Z^d}\chi_{j}(m)\hat{f}(m) e^{\iota2\pi m\cdot x},~j\geq-1.
\end{equation}
Note that \eqref{LPBs} is equivalent to the equality
\begin{equation}\label{LPBs2}
  \Delta_{j}f=\eta_j \ast f,~~j\geq-1,
\end{equation}
where
\[
\eta_j\ast f(\cdot)=\int_{\mathbb T^d} \eta_j (\cdot-x)f(x)dx,~~
\eta_j(x):=\sum_{m\in\mathbb Z^d}\chi_{j}(m) e^{\iota2\pi m \cdot x}.
\]

 Let $\alpha\in\mathbb R,~p,q\in[1,\infty]$. We define the Besov space on $\mathbb T^d$ denoted by $\mathcal B_{p,q}^{\alpha}(\mathbb T^d)$  as
 the completion of $C^\infty(\mathbb T^d)$ with respect to the norm (\cite[Proposition 2.7]{BCD2011})
\begin{equation}\label{B-SPACE}
  \|u\|_{\mathcal B_{p,q}^{\alpha}(\mathbb T^d)}:=(\sum_{j\geq-1} 2^{j\alpha q}\|\Delta_j u\|^q_{L^p})^{1/q},
\end{equation}
with the usual interpretation as $l^{\infty}$ norm in case $q=\infty$.
Note that for $p,q\in[1,\infty)$
\begin{equation*}
\begin{aligned}
   \mathcal B_{p,q}^{\alpha}(\mathbb T^d)&=\{u\in\mathcal S'(\mathbb T^d):~\|u\|^q_{ \mathcal  B_{p,q}^{\alpha}(\mathbb T^d)}<\infty\},\\
  \mathcal B_{\infty,\infty}^{\alpha}(\mathbb T^d)&\varsubsetneq\{u\in\mathcal S'(\mathbb T^d):~\|u\|^q_{\mathcal  B_{\infty,\infty}^{\alpha}(\mathbb T^d)}<\infty\}.
\end{aligned}
\end{equation*}
 Here we choose Besov spaces as completions of smooth functions on the torus,
which ensures that the Besov spaces are separable and has a lot of advantages for our analysis
below.

 In the following we give estimates on the torus for later use.

We recall the  Besov embedding theorems on the torus  (cf. \cite[Theorem 4.6.1]{Tri78}, \cite[Lemma A.2]{GIP13}, \cite[Proposition 3.11,Remark 3.3]{MW2010}).
\begin{lemma} \label{V-SUB-H}{\rm(Besov embedding)}\label{Bes-Smooth}
{\rm (i)} Let $\alpha\leq \beta\in\mathbb R$,  $p,q\in[1,\infty]$. Then $ \mathcal B_{p,q}^{\beta}(\mathbb T^d)$ is continuously embedded   in $ \mathcal B_{p,q}^{\alpha}(\mathbb T^d)$.

{\rm (ii)}~ Let $1\leq p_1\leq p_2\leq \infty$, $1\leq q_1\leq q_2\leq \infty$, and  $\alpha\in \mathbb R$. Then $\mathcal B_{p_1,q_1}^{\alpha}(\mathbb T^d)$ is continuously embedded in  $\mathcal B_{p_2,q_2}^{\alpha-d(1/p_1-1/p_2)}(\mathbb T^d)$.

\end{lemma}

We describe the  Schauder estimates, i.e. the smoothing effect of the heat flow, as measured in Besov spaces (cf. \cite[Propositions 3.11,3.12]{MW2010}, \cite[Lemmas A.7,A.8]{GIP13}).

\begin{lemma}{\rm (Schauder estimates)}\label{Heat-Smooth1}
{\rm(i)} Let $f\in\mathcal B_{p,q}^{\alpha}(\mathbb T^d)$ for some $\alpha\in\mathbb R,~p,q\in[1,\infty]$.
Then for every $\delta>0$,  uniformly over $t>0$
\begin{equation}\label{Heat-Smooth}
\|e^{tA} f\|_{\mathcal B_{p,q}^{\alpha+\delta}(\mathbb T^d)}\lesssim t^{-\frac{\delta}{2}}\|f\|_{\mathcal B_{p,q}^{\alpha}(\mathbb T^d)}.
\end{equation}
{\rm (ii)} Let $\alpha\leq \beta\in\mathbb R$ be such that $\beta-\alpha\leq2$, $f\in \mathcal B_{p,q}^{\beta}(\mathbb T^d)$ and $p,q\in[1,\infty]$. Then uniformly over $t>0$
\begin{equation}\label{Heat-time-reg}
\|({\rm I}-e^{tA}) f\|_{\mathcal B_{p,q}^{\alpha}(\mathbb T^d)}\lesssim t^{\frac{\beta-\alpha}{2}}\|f\|_{\mathcal B_{p,q}^{\beta}(\mathbb T^d)}.
\end{equation}
\end{lemma}

The following multiplicative inequalities play a central role  later and we  treat separately for  cases of positive and negative regularity (cf.  \cite[Corollaries 3.19,3.21]{MW2010}, \cite[Lemma 2.1]{GIP13}).

\begin{lemma}{\rm (Multiplicative inequalities)}\label{Multi-ineq}
{\rm (i)} Let $\alpha>0$ and $p,p_1,p_2\in[1,\infty]$ be such that ${1}/{p}={1}/{p_1}+{1}/{p_2}$. Then
\begin{equation}\label{Multi-ineq-1}
\|fg\|_{\mathcal B_{p,q}^{\alpha}(\mathbb T^d)}\lesssim \|f\|_{ \mathcal B_{p_1,q}^{\alpha}(\mathbb T^d)}\|g\|_{\mathcal B_{p_2,q}^{\alpha}(\mathbb T^d)}.
\end{equation}
{\rm (ii)} Let $\beta>0>\alpha$ be such that $\beta+\alpha>0$, and $p,p_1,p_2\in[1,\infty]$ be such that $ {1}/{p}= {1}/{p_1}+{1}/{p_2}$. Then
\begin{equation}\label{Multi-ineq-2}
\|fg\|_{\mathcal B_{p,q}^{\alpha}(\mathbb T^d)} \lesssim
\|f\|_{\mathcal B_{p_1,q}^{\alpha}(\mathbb T^d)}\|g\|_{\mathcal B_{p_2,q}^{\beta}(\mathbb T^d)}.
\end{equation}
\end{lemma}

\begin{lemma}{\rm \cite[Proposition A.11]{TW2018}}\label{est-P-PN-P}
Let $P_N,N\in\mathbb N$ be defined in (\ref{pro-ope}). Then for every $\alpha\in\mathbb R$, $p,q\in[1,\infty]$ and $\lambda>0$
\begin{equation}\label{est-PN}
\|P_Nf\|_{\mathcal B^{\alpha}_{p,q}(\mathbb T^2)}\lesssim \|f\|_{\mathcal B^{\alpha+\lambda}_{p,q}(\mathbb T^2)},
\end{equation}
\begin{equation}\label{est-P-PN}
\|P_Nf-f\|_{\mathcal B^{\alpha}_{p,q}(\mathbb T^2)}\lesssim\frac{(\log N)^2}{N^{\lambda}}\|f\|_{\mathcal B^{\alpha+\lambda}_{p,q}(\mathbb T^2)}.
\end{equation}
\end{lemma}
Throughout the paper we mainly use   the Besove space with $p=q=\infty$ to study  the equations on $\mathbb T^2$. For the simplicity of natations, for any $\alpha\in\mathbb R$ and $p,q\in[1,\infty)$, let
\begin{equation*}\label{nota}
\mathcal B^{\alpha}_{p,q}:=\mathcal B^{\alpha}_{p,q}(\mathbb T^2),~~\mathcal B^{\alpha}_p:=\mathcal B^{\alpha}_{p,\infty}(\mathbb T^2),~~\mathcal C^{\alpha}:=\mathcal B_{\infty,\infty}^{\alpha}(\mathbb T^2)
\end{equation*}
and we denote their norms by $\|\cdot\|_{\mathcal B^{\alpha}_{p,q}},\|\cdot\|_{\mathcal B^{\alpha}_p}$ and $\|\cdot\|_{{\alpha}}$, respectively.

\section{Stochastic heat equation}\label{SEC3}

As we outlined in the introduction, the solutions to the original stochastic Allen-Cahn equation \eqref{initial-Eq} correspond to the solutions $\bar{Z}$
  of the stochastic heat equation \eqref{ini1-Eq} and the solutions $Y$ of the PDE \eqref{ini2-Eq} with random coefficients derived from the Wick power of $\bar{Z}$. To define the approximation of the stochastic Allen-Cahn equation \eqref{initial-Eq}, we must first construct approximations for $\bar{Z}$ and its Wick power.

In this section, we present a construction of solutions to the stochastic heat equation \eqref{ini1-Eq} and its Wick powers in a Besov space with negative regularity. Initially, we construct solutions to \eqref{ini1-Eq} with an initial value of 0, following the approach in \cite{MW2010,TW20182}. Subsequently, we extend our construction to address the case of an initial value $X_0$ with negative regularity, defining solutions to \eqref{ini1-Eq} and their associated Wick powers, as discussed in \cite{MZ2019}.
We then delve into the analysis of the regularity properties of these solutions and their Galerkin approximations, along with their respective Wick powers.
\subsection{ Wick powers and Galerkin approximation}\label{set-Wick powers}

Let $(\Omega,\mathcal F,\mathbb P)$ be a probability space, $\xi$ is space-time white noise on $\mathbb R\times \T$. Set
\begin{equation*}\label{fil}
\tilde{\mathcal F}_t:=\sigma\big(\{\xi(\phi):~\phi|_{(t,\infty)\times\T}\equiv0,~\phi\in L^2((-\infty,\infty)\times \T)\}\big)
\end{equation*}
for $t>-\infty$ and denote by $(\mathcal F_t)_{t>-\infty}$ the usual augmentation (see \cite[Chapter 1.4]{RY1999}) of the filtration $(\tilde{\mathcal F}_t)_{t>-\infty}$.
For $n=1,2,3$,  consider the multiple stochastic integral given by
\begin{equation}\label{HEAT-MILD-s-23}
Z_{-\infty,t}^{:n:}(\phi):=\int_{\{(-\infty,t]\times\T\}^n}\langle\phi,\Pi_{i=1}^n H(t-s_i,x_i-\cdot)\rangle\xi(\otimes_{i=1}^n \d s_i,\otimes_{i=1}^n \d x_i)
\end{equation}
for every $t\in\mathbb R$ and $\phi\in C^{\infty}(\T)$. Here $H(r,\cdot),~r\neq0$, stands for the periodic heat kernel associated to the generator $A=\Delta-{\rm I}$ on $\T$ given by
\begin{equation}\label{ker-perid}
H(r,x):=\sum_{m\in\Z}e^{-rI_m}e_m(x),~x\in \T,~r\in\mathbb R \setminus \{0\},
\end{equation}
with $I_m:=1+4\pi^2|m|^2$ and $e_m=e^{\iota2\pi m\cdot}$, $m\in\Z$.
$Z^{:n:}_{-\infty,\cdot}$ is called the $n$-th Wick power of $Z_{-\infty,\cdot}$. Let $S_t,t\ge0$  denote  the  semigroup  associated to  $A$ in $L^2(\T)$. Using Duhamel's principle (cf. \cite[Section 2.3]{Ev10}), we have that
\[
   Z_{t}:=Z_{-\infty,t}-S_tZ_{-\infty,0},~t\geq 0,
\]
which solves the linear equation with zero initial condition, i.e.
\begin{equation}\label{HEAT-s}
\left\{
 \begin{aligned}
\t Z&=AZ+\xi~~\text{in~}(0,\infty)\times\T,\\
Z({0})&=0~~\text{on~}\T.
 \end{aligned}
 \right.
\end{equation}
Moreover, we set for $n=1, 2,3$
\begin{equation}\label{Z230twick}
\begin{aligned}
    Z^{:n:}_{t}:=\sum_{k=0}^n\binom{n}{k}(-1)^k\Big(S_tZ_{-\infty,0}\Big)^kZ_{-\infty,t}^{:n-k:},
\end{aligned}
\end{equation}
by letting $ z^{:1:}=z$ and $z^{:0:}=1$.
As discussed in \cite{TW2018}, we continue to define the spatial Galerkin approximation $Z^N$ of $Z$ and its wick powers $(Z^N)^{:n:}$, $n=1,2,3$. Denote by
 \begin{equation} \label{Zst2-3-APP}\begin{aligned}
   Z^N_{-\infty,t}:=P_NZ_{-\infty,t},~~
   ( Z^N_{-\infty,t})^{:2:}:= (Z^N_{-\infty,t})^2-  \mathfrak R^N,~\\
   ( Z^N_{-\infty,t})^{:3:}:= (Z^N_{-\infty,t})^3-  3\mathfrak R^NZ^N_{-\infty,t},~~~~~~~~
\end{aligned}\end{equation}
with the projectors $P_N,N\in\mathbb N$ given in \eqref{pro-ope}, and  the renormalization  constants
\begin{equation*}\label{renor-cosnt}
    \mathfrak R^N:=\|1_{[0,\infty)}P_NH\|_{L^2(\mathbb R\times \T)}^2,
\end{equation*}
which diverges  logarithmically as $N$ goes to $\infty$.
Comparing with \eqref{Z230twick}, we similarly define for  $n=1,2,3$
\begin{equation}\label{ZN1230t-wick}
\begin{aligned}
    (Z_{t}^N)^{:n:}:=\sum_{k=0}^n\binom{n}{k}(-1)^k\Big(S_tZ_{-\infty,0}^N\Big)^k(Z_{-\infty,t}^N)^{:n-k:},t\geq 0.
\end{aligned}
\end{equation}
In particular,
$Z_{t}^N$ solves approximating equation with  initial value zero:
\begin{equation}
\label{ini10-Eq_N}
\left\{
 \begin{aligned}
\t  Z^N&= P_N A Z^N+P_N\xi  \text{~~in~} (0,\infty)\times\T,\\
 Z^N(0)&=0 \text{~~on~} \T.
 \end{aligned}
 \right.
\end{equation}

Now  we combine the initial value part with the Wick powers,   as described in \cite{ MZ2019}. Let $X_0\in\C^{-\alpha},\alpha\in(0,1)$, we set   for $n=1,2,3$
\begin{equation}\label{TZN1230t-wick-ini}
\begin{aligned}
  (\bar Z_t^{N})^{:n:}:&=\sum_{k=0}^n\binom{n}{k} (P_NS_tX_0)^{k} ( Z_t^{N})^{:n-k:},\\
  \bar Z_t ^{:n:}:&=\sum_{k=0}^n\binom{n}{k} ( S_tX_0)^{k}  Z_t^{:n-k:}.
  \end{aligned}
\end{equation}
 The above equalities are well-defined by \eqref{Multi-ineq-2}, and by \eqref{Heat-Smooth} and \eqref{est-PN} which imply that the terms $S_t X_0$,$P_NS_t X_0$ belong to $\C^{\beta}$ for every $\beta>\alpha $.
In particular, when $n=1$, $\bar Z$ is a solution to  \eqref{ini1-Eq} and $\bar Z^N$  solves
\begin{equation}
\label{ini1-Eq_N}
\left\{
 \begin{aligned}
\t \bar Z^N&= P_N A \bar Z^N+P_N\xi  \text{~~in~} (0,\infty)\times\T,\\
\bar Z^N(0)&=P_NX_0 \text{~~on~} \T.
 \end{aligned}
 \right.
\end{equation}

\subsection{Regularity results } \label{SUB3-3}
Let $T>0$ and consider the initial value $X_0\in\mathcal C^{-\alpha}$ with $\alpha\in(0,1)$. Recall that the processes $Z_{-\infty,t}$, $Z_t$, $\bar Z_t$ for $t\in[0,T]$, along with their Galerkin approximations and Wick powers, are defined in Section \ref{set-Wick powers}.
In this subsection, we establish that these processes are well-defined elements in Besov spaces of negative regularity. Similar results have been shown in \cite[Lemma 3.4]{RZZ2017}, \cite[Proposition 7.4]{TW2018}, and \cite[Propositions 2.2, 2.3]{TW20182}.
Furthermore, we derive the time regularity properties of the aforementioned processes in the same space, which are essential for further estimates.

\begin{lemma}\label{Tmodefi-Z}
Let $\alpha\in(0,1)$ and $p\ge 2$. Then for every $n=1,2,3$ and  $N\in\mathbb N$, the processes $Z_{-\infty,\cdot}^{:n:}, (Z^{N}_{-\infty,\cdot})^{:n:}$   defined by  \eqref{HEAT-MILD-s-23} and \eqref{Zst2-3-APP} belong to  $ C([0,T]; \C^{-\alpha})$ $\mathbb P$-a.s. Moreover, we have
\begin{equation}\label{TZN^nt-lp}
 \sup_{N\in\mathbb N}   \E \sup_{0\leq t\leq T}  \Big\{  \|Z_{-\infty,t}^{:n:}\|_{ {-\alpha}}^p,
      \|(Z^{N}_{-\infty,t})^{:n:}\|_{ {-\alpha}}^p\Big \}<\infty.
\end{equation}
As a consequence, let the processes $Z^{:n:},(Z^{N})^{:n:},n=1,2,3$ be defined by \eqref{Z230twick} and \eqref{ZN1230t-wick}, for  every $\alpha'>0$
\begin{equation}\label{Test-ZN-HOLD}
   \sup_{N\in\mathbb N} \E \sup_{0\leq t\leq T}  t^{(n-1){\alpha}'p}  \Big\{
         \|Z_{t}^{:n:}\|_{ {-\alpha}}^p,  \|(Z^{N}_{t})^{:n:}\|_{ {-\alpha}}^p
        \Big \}<\infty.
\end{equation}
\end{lemma}
\begin{proof} See \cite[Lemmas 3.2,3.3]{MZ2019}.
\end{proof}
\begin{lemma}
Let  $X_0\in\C^{-\alpha}$ with $\alpha\in(0,1)$, and the processes $\bar Z ^{:n:}$,$(\bar Z^N)^{:n:}$,$n=1,2,3$ be defined by \eqref{TZN1230t-wick-ini}. Then  for  every  $p\ge 2$ and $\kappa>0$
\begin{equation}\begin{aligned}\label{Test-ZN-HOLD-INI2}
   & \E \sup_{0\leq t\leq T}t^{(n-1)(\alpha+\kappa)  p}    \|(\bar Z_{t})^{:n:}\|_{-\alpha}^p<\infty, \\
      \sup_{N\in\mathbb N} & \E \sup_{0\leq t\leq T}t^{\{(n-1)\alpha +\kappa\} p}
      \|(\bar Z^N_{t})^{:n:}\|_{-\alpha}^p  <\infty.
 \end{aligned}\end{equation}
\end{lemma}
\begin{proof}
See \cite[Theorem 3.5]{MZ2019}.
\end{proof}
Below we mainly discuss the time regularity properties of $(\bar{Z}^N)^{:n:}$, $n=1,2,3$,   which are essential for the subsequent estimation. In view of \eqref{ZN1230t-wick} and \eqref{TZN1230t-wick-ini}, our initial focus is  on discussing the relative properties of $(Z_{-\infty,\cdot}^N)^{:n:}$ and $(Z^N)^{:n:}$.

\begin{lemma}\label{TZtN-ZsN0}
Let the processes $(Z^{N}_{-\infty,\cdot})^{:n:}$,$n=1,2,3$  be defined by  \eqref{Zst2-3-APP}.  Then for every   $\alpha\in (0,1)$, $p\ge2$ and $\delta\in(0,\frac\alpha{2n})$
\begin{equation}\label{TZNM^nt-sp}
  \sup_{N\in\mathbb N}   \E \sup_{0\leq  s<t\leq T} {(t-s)^{(\delta-\frac{\alpha}{2n})p}} \|(Z^{N}_{-\infty,t})^{:n:}-(Z^{N}_{-\infty,s})^{:n:}\|_{ {-\alpha}}^p  <\infty.
\end{equation}
Moreover, let the processes $(Z^{N})^{:n:},n=1,2,3$ be defined by \eqref{ZN1230t-wick}. 
For every $\epsilon>0$
\begin{equation}\label{Test-ZNT-S-HOLD}
       \sup_{N\in\mathbb N}   \E \sup_{0\leq s<t\leq T}  \frac { s^{(\epsilon+
       \frac\alpha {2n})(n-1)p} } {(t-s)^{(\frac{\alpha}{2n}-\delta)p}}
          \|(Z^{N}_{t})^{:n:}-(Z^{N}_{s})^{:n:}\|_{ {-\alpha}}^p   <\infty.
\end{equation}
\end{lemma}
\begin{proof}
 Due to  \cite[(3.11)]{MZ2019} we have for any $s,t\in[0,T]$, $x\in\mathbb T^2 $
\begin{equation}\label{nnnNM-ZsZt}
\begin{aligned}
 & \E\Delta_{j}(Z^{N}_{-\infty,s})^{:n:}(x)\Delta_{j}(Z^{N}_{-\infty,t})^{:n:}(x)
    \backsimeq
   n! \sum_{\substack{m_1\in\mathcal A_{2^{j}},\\ m_1\in\mathbb Z^2}} \sum_{\substack{|m_i|\leq N,\\ m_i\in\mathbb Z^2,i=2,\ldots,n}}
     \prod^n_{i=1}
  \frac{  e^{-|t-s|I_{m_i-m_{i-1}}}} {2 I_{m_i-m_{i-1}}}
\end{aligned}
\end{equation}
with  the convention that $m_0 = 0$, $I_m=1+4\pi^2|m|^2$, $e_m=e^{\iota2\pi m\cdot}$ for $m\in\Z$, and $\Delta_j,j\ge-1$ given in \eqref{LPBs}. Then for  $K_\gamma(m):=\frac{1}{(1+|m|^2)^{1-\gamma}},\gamma\in(0,1)$, it can be shown that 
\begin{equation*}
\label{s neq t}
 \E\big|\Delta_{j}(Z^{N}_{-\infty,s})^{:n:}(x)-\Delta_{j}(Z^{N}_{-\infty,t})^{:n:}(x)\big|^2
\lesssim  |s-t|^{\gamma}
   \sum_{m\in\mathcal A_{2^{j}},m\in\mathbb Z^2} K^\gamma \star^{n}_{\leq N}
   K^\gamma (m),
\end{equation*}
where $K^\gamma \star^{n}_{\leq N}
   K^\gamma$ is defined in  \eqref{<Nn}.
Using Lemma  \ref{KENEL-ES}, we have for  any $\lambda>0$ and $\gamma\in(0,1/ n)$
\begin{equation*}\begin{aligned}
 \E \big|\Delta_{j}(Z^{N}_{-\infty,s})^{:n:}(x)-\Delta_{j}(Z^{N}_{-\infty,t})^{:n:}(x)\big|^2
 \lesssim&
   \sum_{m\in\mathcal A_{2^{j}},m\in\mathbb Z^2} \frac{ |t-s|^{\gamma}}{(1 + |m|^2)^{1-n\gamma}}
\end{aligned}
\end{equation*}
 uniformly for  $x\in\T$, $s,t\in[0,T]$, $j\geq -1$  and  $N\in \mathbb N$. Considering that $|m|\lesssim 2^{j}$ for $m\in\mathcal A_{2^j}$ in \eqref{subset},
 together with \eqref{iTO-iEQ} we  further have   for any $\gamma\in(0, 1 /n )$ and $p\geq 2$
\begin{equation*}\label{Ttime-reg-Z2}
\E \big|\Delta_{j}(Z^{N}_{-\infty,s})^{:n:}(x)-\Delta_{j}(Z^{N}_{-\infty,t})^{:n:}(x)\big|^p
 \lesssim |s-t|^{\frac{\gamma p}{2}}  2^{j n\gamma p },
\end{equation*}
where the constants we omit are independent of $N$ and variables $s,t,x$. Then by \eqref{B-SPACE}
and the embedding $\mathcal B_{p,p}^{-\alpha+\frac{2}{p}}\hookrightarrow\C^{-\alpha}$ with any $\alpha>\frac{2}{p}$, the Kolmogorov's criterion implies that for any $\alpha>{n\gamma}+\frac{2}{p}$ and any $\tilde{\delta}\in(0,\frac{\gamma}{2}-\frac{1}{p})$,
$$
\sup_{0\leq s\leq t\leq T}|s-t|^{-\tilde{\delta}}\|(Z^{N}_{-\infty,t})^{:n:}-(Z^{N}_{-\infty,s})^{:n:}\|_{ {-\alpha}}^p <\infty, \mathbb{P}-a.s..
$$
Since $p$ could be chosen as large as possible, the above estimate holds for any $\alpha>n\gamma$ and any $\tilde{\delta}\in (0,\frac{\alpha}{2n})$. Therefore, we deduce \eqref{TZNM^nt-sp} for all $p\geq 2$ by setting $\delta:=\frac{\alpha}{2n}-\delta$.

We continue to prove \eqref{Test-ZNT-S-HOLD}. Let $\delta\in(0,\frac \alpha{2n})$ and $\epsilon>0$ small enough.  Applying  Lemma \ref{Multi-ineq}  to \eqref{ZN1230t-wick}, we have for $n=2,3$
\begin{equation*} \begin{aligned}
\|(Z_{t}^{N})^{:n:}&- (Z_{s}^{N})^{:n:}\|_{-\alpha}
 \lesssim       \sum_{k=1}^{n-1}  \Big\{ \|S_tZ_{-\infty,0}^{N}\|_{\frac{n-k}{n}\alpha+\epsilon}^{k} \|(Z_{-\infty,t}^{N})^{:n-k:}-(Z_{-\infty,s}^{N})^{:n-k:}\|_{\frac{k-n}{n}\alpha}
      \\ &    +   \|(S_t-S_s)Z_{-\infty,0}^{N}\|_{2\epsilon}  \|S_sZ_{-\infty,0}^{N}\|_{2\epsilon}^{k-1}  \|(Z_{-\infty,s}^{N})^{:n-k:}\|_{-\epsilon} \Big\} \\
   &   +  \|(Z_{-\infty,t}^{N})^{:n:}-(Z_{-\infty,s}^{N})^{:n:}\|_{-\alpha}
        +   \|(S_t-S_s)Z_{-\infty,0}^{N}\|_{-\epsilon}  \|S_sZ_{-\infty,0}^{N}\|_{2\epsilon}^{n-1},
\end{aligned}
\end{equation*}
where, on the right-hand side, according to Lemma  \ref{Heat-Smooth1} for every $k=1,...,n-1$ we have
\[
   t^{k\epsilon +\frac{k(n-k)}{2n}\alpha}     \|S_tZ_{-\infty,0}^{N}\|^k_{\frac{n-k}{n}\alpha+\epsilon}\lesssim \| Z_{-\infty,0}^{N}\|^k_{-\epsilon},
 \]
\[
       (t-s)^{\delta-\frac{\alpha}{2n}}s^{\frac\alpha{2n}-\delta+{3k\epsilon}}   \|(S_t-S_s)Z_{-\infty,0}^{N}\|_{2\epsilon} \|S_sZ_{-\infty,0}^{N}\|_{2\epsilon}^{k-1}
      \lesssim   \| Z_{-\infty,0}^{N}\|_{-\epsilon}^k,
\]
\[
       (t-s)^{\delta-\frac{\alpha}{2n}} s^{\frac\alpha{2n}-\delta+3(n-1)\epsilon} \|(S_t-S_s)Z_{-\infty,0}^{N}\|_{-\epsilon}
       \|S_sZ_{-\infty,0}^{N}\|_{2\epsilon}^{n-1}
      \lesssim    \| Z_{-\infty,0}^{N}\|_{-\epsilon}^n.
\]
Therefore, \eqref{Test-ZNT-S-HOLD} follows from  \eqref{TZN^nt-lp}, \eqref{TZNM^nt-sp} and Cauchy-Schwarz's inequality.

    When  $n=1$, we   estimate  that
\begin{equation*}
\|Z_{t}^{N}- Z_{s}^{N}\|_{-\alpha}
 \lesssim
     \|Z_{-\infty,t}^{N}-Z_{-\infty,s}^{N}\|_{-\alpha}
        +   \|(S_t-S_s)Z_{-\infty,0}^{N}\|_{-\alpha}  ,
\end{equation*}
where, by using similar procedures
\[
       (t-s)^{\delta-\frac{\alpha}{2}} \|(S_t-S_s)Z_{-\infty,0}^{N}\|_{-\alpha}
      \lesssim    \| Z_{-\infty,0}^{N}\|_{-2\delta}.
\]
Hence \eqref{Test-ZNT-S-HOLD} holds with $n=1$ by \eqref{TZN^nt-lp}.
\end{proof}

Now  we present the time regularity properties  of $(\bar Z^{N})^{:n:}$ for  $n=1,2,3$.
\begin{theorem}
\label{Z-INI} Let  $X_0\in\C^{-\alpha}$ with $\alpha\in(0,1)$, and the processes $(\bar Z^N)^{:n:}$,$n=1,2,3$ be defined by \eqref{TZN1230t-wick-ini}.  Then for any $p\ge 2$, $\kappa>0$ and $\delta\in(0,\frac \alpha{2n})$ 
\begin{equation}\label{TZTN-ZN-n-iniZMNt-s}
  ~ \sup_{N\in\mathbb N} \E \sup_{0\leq s< t\leq T} \frac{ s^{\mu_n p} }   { (t-s)^{(\frac{\alpha}{2n}-\delta)p}}
    \|(\bar Z^{N}_{t})^{:n:}-(\bar Z^{N}_{s})^{:n:}\|_{ {-\alpha}}^p <\infty.
\end{equation}
Here $\mu_1=\frac\alpha{2}-\delta+\kappa$, and $\mu_n=\frac{n^2+1}{2n}\alpha-\delta+\kappa$ for $n=2,3$.
\end{theorem}
\begin{proof}  Let $\delta\in(0,\frac \alpha{2n})$ and $\epsilon>0$ small enough.  For $n=2,3$, utilizing  Lemma \ref{Multi-ineq} and   \eqref{est-PN}  in \eqref{TZN1230t-wick-ini} we obtain
\begin{equation*} \begin{aligned}
\|(\bar Z_{t}^{N})^{:n:}&- (\bar Z_{s}^{N})^{:n:}\|_{-\alpha}
 \lesssim    \sum_{k=1}^{n-1}  \Big\{ \|S_tX_0\|_{\frac{n-k}{n}\alpha+\epsilon}^{k} \|(Z_{t}^{N})^{:n-k:}-(Z_{s}^{N})^{:n-k:}\|_{\frac{k-n}{n}\alpha}
      \\ &                  +  \|(S_t-S_s)X_0\|_{2\epsilon}  \|S_sX_0\|_{2\epsilon}^{k-1}  \|(Z_{s}^{N})^{:n-k:}\|_{-\epsilon} \Big\}\\
      &+ \|(Z_{t}^{N})^{:n:}-(Z_{s}^{N})^{:n:}\|_{-\alpha}   +  \|(S_t-S_s)X_0\|_{0}  \|P_NS_sX_0\|_{2\epsilon}^{n-1},
\end{aligned}
\end{equation*}
 on the right-hand side of which, by Lemma \ref{Heat-Smooth1} for every $k=1,...,n-1$
 \[   t^{\frac k2\alpha+\frac{k(n-k)}{2n}\alpha+k\epsilon} \|S_tX_0\|^k_{\frac{n-k}{n}\alpha+\epsilon}
 \lesssim \| X_0\|^k_{-\alpha},
 \]
\[
  (t-s)^{\delta-\frac{\alpha}{2n}} s^{\frac k2\alpha+\frac\alpha {2n}-\delta+2k\epsilon}  \|(S_t-S_s)X_0\|_{2\epsilon}\|S_sX_0\|_{2\epsilon}^{k-1} \lesssim   \| X_0\|^k_{-\alpha},
\]
\[
     (t-s)^{\delta-\frac{\alpha}{2n}}s^{\frac n2\alpha +\frac\alpha{2n}-\delta+2(n-1)\epsilon}    \|(S_t-S_s)X_0\|_{0}\|S_sX_0\|_{2\epsilon}^{n-1}
     \lesssim \| X_0\|_{-\alpha}^n.
\]
Therefore, \eqref{TZTN-ZN-n-iniZMNt-s} holds for $n=2,3$ by \eqref{Test-ZNT-S-HOLD}, \eqref{Test-ZN-HOLD} and Cauchy-Schwarz's inequality.  Finally, for $n=1$  we similarly estimate that
\begin{equation*}
\|\bar Z_{t}^{N}- \bar Z_{s}^{N} \|_{-\alpha}
 \lesssim      \|Z_{t}^{N}-Z_{s}^{N}\|_{-\alpha}   +  \|P_N(S_t-S_s)X_0\|_{-\alpha},
\end{equation*}
where for  any $\epsilon>0$
\[
     (t-s)^{\delta-\frac{\alpha}{2}}s^{\frac\alpha{2}-\delta+\epsilon}    \|P_N(S_t-S_s)X_0\|_{-\alpha}
     \lesssim \| X_0\|_{-\alpha}.
\]
 \eqref{TZTN-ZN-n-iniZMNt-s} holds for $n=1$ immediately.
\end{proof}

\section{Proof of main result }\label{Pathwise-error}
In this section, we will develop a space-time fully discrete scheme $X^{N,M}$ in \eqref{nNMX} for the solution $X$ to \eqref{initial-Eq}. Here, $N$ is linked to the spectral Galerkin approximation previously introduced, and $M$ is linked to the temporal discretization. To achieve this, we consider the processes ${\bar Z}^{:n:},({\bar Z}^N)^{:n:}$, $n=1,2,3$, $N\in\mathbb N$ introduced in Section \ref{SEC3}. Then 
we interpret (\ref{ini2-Eq}) in the mild sense, i.e.  $Y$ solves (\ref{ini2-Eq}) if for every $t\geq 0$
\begin{equation} \label{Tmild-2}
 Y_t=\int^t_0 S_{t-s}\Psi(Y_s,\underline{\bar Z}_s)ds,
\end{equation}
where
$\Psi$ is given by \eqref{wick-YZ}. Existence and uniqueness of the  mild solution \eqref{Tmild-2} to equation \eqref{ini2-Eq}   has been widely discussed (e.g. \cite[Theorem 3.10]{RZZ2017}, \cite[Theorem 6.2]{MW2010}). It is more convenience to construct a space-time approximation $Y^{N,M}$ via \eqref{Tmild-2} by considering the spectral Galerkin approximation $({\bar Z}^N)^{:n:}$ (see \eqref{Tmild-2-NM-fudis} below for details).
Then by the uniform a-priori bounds of   $Y^{N,M}$ in \eqref{Tmild-2-NM}, we obtain the error estimate between the nonlinear term $Y$ in \eqref{Tmild-2} and its  space-time approximation $Y^{N,M}$ in \eqref{Tmild-2-NM}. Finally, we present the main result, i.e. the   convergence rates in time and space for full-discrete approximations of   \eqref{initial-Eq}.

We also mention that Tsatsoulis and  Weber in \cite{TW2018,TW20182} split $X$ into
$X=\bar Y+ Z$ instead of \eqref{sum}, where $Z$  satisfies \eqref{HEAT-s} and $\bar Y$ solves   the following equation  instead of \eqref{ini2-Eq}
 \begin{equation}
\label{ini2-Eqbar}
\left\{
 \begin{aligned}
\t \bar Y&=A \bar Y+\Psi(\bar Y,\underline{ Z}) \text{~~in~} (0,\infty)\times\T,\\
\bar Y(0)&=X_0\text{~~on~} \T,
 \end{aligned}
 \right.
\end{equation}
 in the mild sense, i.e.  for every $t\ge 0$
\begin{equation*} \label{Tmild-2bar}
 \bar Y_t=S_tX_0+\int^t_0 S_{t-s}\Psi(\bar Y_s,\underline{ Z}_s)ds.
\end{equation*}
They obtained local existence and uniqueness of the above mild solution  $\bar Y$ in a Besov space $\mathcal C^\beta$ ($\alpha<\beta<2-\alpha$) with the norm $\sup_{t\in[0,T]}t^{\gamma}\|\cdot\|_{\beta}$. The coefficient $\gamma\in(\frac{\alpha+\beta}{2},\frac 13-\frac\beta 6)$ is used to measure the blow-up of $\|\bar Y_t\|_{\beta}$ for $t$ close to 0 (see \cite[Theorem 3.9]{TW20182}). Actually,  \eqref{ini2-Eq} is equivalent to \eqref{ini2-Eqbar}. More precisely,  $\bar Y$ is a solution to  \eqref{ini2-Eqbar}  if and only if $Y:=\bar Y-S_{\cdot}X_0$ is a solution to  \eqref{ini2-Eq} (refer to \cite[Theorems 3.9,4.8]{RZZ2017}).
Based on the above discussion, using a fixed  point argument we present  the regularity property of $Y$ in the Besov space $\C^\beta$ with the norm $\|\cdot\|_\beta$, instead of  the norm $t^\gamma\|\cdot\|_\beta$ on the same space (see \cite[Theorem 3.9]{TW20182}).

Assume that   the positive coefficients $\alpha<\beta<1$  satisfy
	\begin{equation}\begin{aligned}
			\label{coef}
				\frac{ 5\alpha+\beta}{2} <1 .
	\end{aligned}\end{equation}

\begin{theorem}    \label{T-Y-new-Th}
Let $p\ge2$ and  $\alpha, \beta$ satisfy \eqref{coef}.
Then
 \begin{equation}\label{est-Y-EHOLD}
\E \sup_{t\in[0,T]}\|Y_{t} \|^p_{\beta}<\infty.
 \end{equation}
\end{theorem}
\begin{proof} With the regularity property of $\bar Z$ described in \eqref{Test-ZN-HOLD-INI2}, the procedure is essentially the same as in the proof of \cite[Theorem 3.9]{TW20182}, if we  replace $Z$ by $\bar Z$ and set $\gamma=0$ and the initial value $x=0$. We omit the details.
\end{proof}
\subsection{Space-time full discretization}\label{SEC4-11}
In this subsection we propose a space-time
approximation $X^{N,M}$ of \eqref{initial-Eq} by  tamed exponential Euler discretization in time and spectral  Galerkin method in space.

Let $T>0$ and $M\in\mathbb N$,  we construct a uniform mesh on $[0,T]$ with $\tau=T/M$ being the time stepsize, and define $$t_k:=k\tau,~~k=0,1,...,M.$$
Inspired by \cite{W20}, we propose a space-time full discretization $Y^{N,M}$ of $Y$ as  $Y_{t}^{N,M}\equiv0,t\in[0,\tau]$, and for every $m=1,...,M-1$
\begin{equation}\label{Tmild-2-NM-fudis}
	 Y_{t_{m+1}}^{N,M}=S(\tau) Y_{t_{m}}^{N,M}
	 +\int^{t_{m+1}}_{t_{m}}\frac{P_N S_{t_{m+1}-s} \Psi(Y_{t_{m}}^{N,M},
	 	\underline{\bar Z}_{t_{m}}^{N} )}
	{1+\tau \|\Psi(Y_{t_{m}}^{N,M},\underline{\bar Z}_{t_{m}}^{N} )\| _{-\alpha}  } ds
\end{equation}
with  $\bar Z^{N}_0=P_NX_0$.  Define
\[
{\lfloor t \rfloor}_\tau:=t_m, \text{~~~~for $t$ in $[t_m,t_{m+1})$, $m\in\{0,1,...,M-1\}$}.
\]
Then we  introduce a continuous version of the fully discrete version  \eqref{Tmild-2-NM-fudis} as
\begin{equation}\label{Tmild-2-NM}\begin{aligned}
 Y_t^{N,M}=&\int^{t\vee \tau}_\tau\frac{P_N S_{t-s}\Psi(Y_{{\lfloor s\rfloor}_\tau}^{N,M},\underline{\bar Z}_{{\lfloor s\rfloor}_\tau}^{N} )}
 {1+\tau \|\Psi(Y_{{\lfloor s\rfloor}_\tau}^{N,M},\underline{\bar Z}_{{\lfloor s\rfloor}_\tau}^{N} )\| _{-\alpha}}  ds,~~t\in[0,T].
       \end{aligned}
\end{equation}
 Finally, the  space-time full discretizations of \eqref{sum}   are constructed as
\begin{equation}\begin{aligned}\label{nNMX}
 X_t^{N,M}=&\int^{t\vee \tau}_\tau\frac{P_N S_{t-s} \Psi(Y_{{\lfloor s\rfloor}_\tau}^{N,M},\underline{\bar Z}_{{\lfloor s\rfloor}_\tau}^{N} )}
 {1+\tau \|\Psi(Y_{{\lfloor s\rfloor}_\tau}^{N,M},\underline{\bar Z}_{{\lfloor s\rfloor}_\tau}^{N} )\| _{-\alpha}}  ds  +\bar Z^N_t,~t\in[0,T].
       \end{aligned}
\end{equation}
Indeed, on the right-hand side of  \eqref{Tmild-2-NM} and \eqref{nNMX} we integrate from  $\tau$ instead of 0 as given in \cite{W20}. It is because that ${\lfloor s\rfloor}_\tau=0$ for any $s\in[0,\tau)$, and  the term $S_{{\lfloor s\rfloor}_\tau}X_0$ still remains in a Besov space of negative space, which leads that   the terms $(\bar{Z}^N_{{\lfloor s\rfloor}_\tau})^{:n:}$,$n=2,3$ defined in \eqref{TZN1230t-wick-ini} are not well-defined  appearing in the function $\Psi$.
\subsection{A priori bounds for  the approximations}\label{SEC4-1}
The aim of this subsection is to prove a priori  bounds  for the full-discrete approximations $Y^{N,M}$, $N,M\in\mathbb N$ defined in \eqref{Tmild-2-NM}.

\begin{theorem}    \label{Test-Y-HOLD-Th}
Let $p\ge2$ and $\alpha, \beta $ satisfy \eqref{coef}.
Then   $Y^{N,M}\in C([0,T];\C^\infty)$ and
 \begin{equation}\label{Test-Y-EHOLD}
\sup_{N,M\in\mathbb N}\E \sup_{t\in[0,T]}  \|Y_{t}^{N,M}\|^p_{\beta}<\infty.
 \end{equation}
\end{theorem}
Let $K>0$, for $m=0,1,...,M$,  set
\begin{equation}\label{set-bdd} \Omega^{K}_{N,m}:=\Big\{\sup_{j\in\{0,1,...,m\}}  \|Y_{t_j}^{N,M}\|_{\beta} \le K \Big\}.
\end{equation}
 To prove Theorem \ref{Test-Y-HOLD-Th}, we initially establish uniform  a prior bounds of $Y^{N,M}$ on the subset ${\Omega_{N,m}^K}$,  and then we extend the result to the entire set. In the following we denote by $\bar \Omega$ and ${1}_\Omega$ the complement and indicator function of a set $\Omega$. It is known that $1_{\Omega_{N,m}^K}$ is ${\mathcal F}_{t_m}$ adapted.
\begin{lemma}    \label{Test-Yset-Ti-HOLD-Th}
	Let $p\ge2$, $\alpha,\beta$  satisfy \eqref{coef}, $\delta_n\in(0,\frac \alpha{2n})$, $n=1,2,3$, and $K:=\tau^{-\mu}$ for any $\mu\in(0,\min\{ \frac16,\frac\alpha 4-\frac{\delta_1}{2},\frac\alpha4-\delta_2,\frac15- \frac{5\alpha+\beta}{10}\})$.
	Then   $Y^{N,M}\in C([0,T];\C^\infty)$ and
	\begin{equation}\label{Test-Yset-ti-EHOLD}
		\sup_{N,M\in\mathbb N}\E \sup_{m\in\{0,...,M\}}\sup_{t\in[0,t_m]}1 _{\Omega^K_{{N,m-1}}} \|Y_{t}^{N,M}\|^p_{\beta}<\infty
	\end{equation}
where we set $1_{\Omega^K_{{N,{-1}}}}=1$.
\end{lemma}

\begin{proof}
 We   introduce a process given by
\begin{equation}\label{TmildVY-NM}\begin{aligned}
 V_t^{N,M}=&\int^{t\vee\tau}_\tau\Big\{\frac{P_N S_{t-s}\Psi(Y_{\lfloor s\rfloor_\tau}^{N,M},\underline{\bar Z}_{\lfloor s\rfloor_\tau}^{N} )}
 {1+\tau \|\Psi(Y_{\lfloor s\rfloor_\tau}^{N,M},\underline{\bar Z}_{\lfloor s\rfloor_\tau}^{N} )\|_{-\alpha}  }
- {P_N S_{t-s}\Psi(Y_{ s }^{N,M},\underline{\bar Z}_{ s }^{N} )}\Big\}
   ds
\end{aligned}
\end{equation}
for $t\in[0,T]$.
 Then  we have decomposition
\begin{equation}\label{finalyyv}
  {Y}_t^{N,M}={V}_t^{N,M}+\hat{Y}_t^{N,M},
\end{equation}
where the process $\hat{Y}^{N,M}$ satisfies
\begin{equation}\label{Tmild-Yhat-NM}\begin{aligned}
 \hat{Y}_t^{N,M}= \int^{t\vee\tau}_\tau{P_N S_{t-s}
     \Psi(\hat{Y}_s^{N,M}+{V}_s^{N,M},\underline{\bar Z}_{ s }^{N} )}ds,~t\in[0,T].
\end{aligned}
\end{equation}
According to the identity \eqref{finalyyv}, we bound  $V_t^{N,M}$ and ${\hat Y}_t^{N,M}$ separately. 

To begin with,  let $R>0$, $N\in\mathbb N$ and $\delta_n\in(0,\frac \alpha{2n})$, $n=1,2,3$,  we set the stopping time as
 \begin{equation}\label{stp01}\begin{aligned}
	\sigma^N_R:=&\inf \Big\{t\leq T: \max\Big(
t^{(m-1)\alpha+\kappa}
	\|(\bar Z_{t}^{N})^{:m:}\|_{-\alpha},~m=1,2,3,~~ \\
	& \frac{  1_{\{t>{\lfloor t\rfloor}_\tau\}}{\lfloor t\rfloor}_\tau^{\mu_n  } }
  { (t-{{\lfloor t\rfloor}_\tau})^{\frac{\alpha}{2n}-\delta_n}}	\|(\bar Z^{N}_{t})^{:n:}-(\bar Z^{N}_{{{\lfloor t\rfloor}_\tau}})^{:n:}\|_{ {-\alpha}},~n=1,2,3\Big)\ge R\Big\} \wedge T
\end{aligned}	
\end{equation}
with $\mu_1=\frac{\alpha}{2}-\delta_1$, $\mu_n=\frac{n^2+1}{2n}\alpha-\delta_n+\kappa$, $n=2,3$,  by  setting $\inf \emptyset =\infty$.\\
{\textbf{Step 1: Estimate of $V_t^{N,M}$.}}  We split $V_t^{N,M}$ as $V_t^{N,M}=L_1+L_2$ with
\begin{equation}\label{L1L2}\begin{aligned}
		L_1:=&\int^{t\vee\tau}_\tau  P_N S_{t-r}\Big\{\Psi(Y_{ \lfloor r \rfloor_\tau}^{N,M},\underline{\bar Z}_{\lfloor r \rfloor_\tau}^{N} )
		- \Psi(Y_{   r  }^{N,M},\underline{\bar Z}_{  r  }^{N} )\Big\}dr,\\
		L_2:=
		&\int^{t\vee\tau}_\tau P_N S_{t-r} \Psi(Y_{\lfloor r\rfloor_\tau}^{N,M},\underline{\bar Z}_{\lfloor r\rfloor_\tau}^{N} ) \frac{ -\tau \| \Psi(Y_{\lfloor r\rfloor_\tau}^{N,M},\underline{\bar Z}_{\lfloor r\rfloor_\tau}^{N} )\| _{-\alpha} }
		{1+\tau \|\Psi(Y_{\lfloor r\rfloor_\tau}^{N,M},\underline{\bar Z}_{\lfloor r\rfloor_\tau}^{N} )\|  _{-\alpha}}  dr.
	\end{aligned}
\end{equation}
 For the simplification of the notations,   we use the decomposition
\begin{equation*} \label{Nota-F-De}
	\Psi(u,\underline { z})=F(u)+\widetilde{\Psi}(u,\underline{ z}),
\end{equation*}
with $u\in \C^\beta, \underline{ z}=( z,z^{:2:}, z^{:3:})$ that
\begin{equation*} \label{Nota-F}
	F(u):=\sum_{i=0}^3a_iu^i ,~
	\widetilde{\Psi}(u,\underline { z}):=\sum_{i=1}^3a_i z^{:i:}+3a_3(u^2 z+u z^{:2:})+2a_2u z .
\end{equation*}
Under the assumption that $0<\alpha<\beta,\alpha+\beta<2$ and applying Lemma  \ref{Multi-ineq} and Young's inequality,  for any $u\in\C^\beta,\underline{ z}=(z,z^{:2:},z^{:3:})$ with $z^{:n:}\in\C^{-\alpha}$,$n=1,2,3$, we easily have that ${F}(u)\in\C^\beta, \widetilde{\Psi}(u,\underline z)\in\C^{-\alpha}$ with
\begin{equation}\begin{gathered}\label{F-PSI}
		\|{F}(u)\|_{{\beta}}\lesssim 1+\|u\|^3_{{\beta}},\\
		\|\widetilde{\Psi}(u,\underline{z})\|_{-\alpha}\lesssim    1+\big(1+\|u\|^2_{\beta}\big)\| z\|_{-\alpha}
		+\big(1+\|u\|_{\beta}\big)\|z^{:2:}\|_{-\alpha}+\| z^{:3:}\|_{-\alpha},
	\end{gathered}
\end{equation}
and that for $v\in\C^\beta, \underline{w}=(w,w^{:2:},w^{:3:})$ with $w^{:n:}\in\C^{-\alpha}$, $n=1,2,3$
\begin{equation}\begin{gathered}\label{F-PSI2}
		\|{F}(u)-F(v)\|_{{\beta}}\lesssim \|u-v\|_{{\beta}}( 1+\|u\|^2_{{\beta}}+\|v\|^2_{{\beta}}),\\
		 \|\widetilde{\Psi}(u,\underline {z})-\widetilde{\Psi}(v,\underline{w})\|_{-\alpha}
		\lesssim \big\{ (1+\|u\|_{{\beta}} +\|v\|_{{\beta}})
		\| z\|_{-\alpha}+\|z^{:2:}\|_{-\alpha} \big\}\|u-v\|_{{\beta}} \\
		+(1 + \|v\|_{{\beta}}^2) \| z- w\|_{-\alpha}
		+ (1+\|v\|_{{\beta}} )\| z^{:2:}-w^{:2:}\|_{-\alpha}
		+\| z^{:3:}- w^{:3:}\|_{-\alpha}.
	\end{gathered}
\end{equation}
{\textbf{Item $L_1$.~}} Let $t\in[0,t_m\wedge \sigma^N_R)$.
Making use of   \eqref{est-PN}  and  \eqref{Heat-Smooth} we have
\begin{equation*}
	\begin{aligned}
\|L_1\|_{\beta} \lesssim
	\int^{t\vee\tau}_\tau &  \Big\{(t-r)^{-\frac \kappa2}\|F(Y_{r}^{N,M})-F(Y_{\lfloor r\rfloor_\tau}^{N,M})\|_\beta
	+(t-r)^{-\frac {\kappa+\alpha+\beta}2}
	\\&\cdot\|\widetilde{\Psi}(Y_{r}^{N,M},\underline {{\bar Z}}_r^{N})	
	-\widetilde{\Psi}(Y_{\lfloor r\rfloor_\tau}^{N,M},\underline {{\bar Z}}_{\lfloor r\rfloor_\tau}^{N})\|_{-\alpha}\Big\}dr,
	\end{aligned}
\end{equation*}
and by \eqref{F-PSI2} and Young's inequality for further estiamte that
\begin{equation}\label{L1EST}
	\begin{aligned}
&	\|L_1\|_{\beta} \lesssim	\int^{t\vee\tau}_\tau  \Big\{(t-r)^{-\frac \kappa2}\|Y_{r}^{N,M}-Y_{\lfloor r\rfloor_\tau}^{N,M}\|_\beta\big(1+\|Y_{r}^{N,M}\|^2_\beta+\|Y_{\lfloor r\rfloor_\tau}^{N,M}\|^2_\beta \big)\\
		&+(t-r)^{-\frac {\kappa+\alpha+\beta}2}
		\|Y_{r}^{N,M}-Y_{\lfloor r\rfloor_\tau}^{N,M}\|_\beta
		\sum_{k=1}^2\big(1+\|Y_{r}^{N,M}\|^{2-k}_\beta +
		\|Y_{\lfloor r\rfloor_\tau}^{N,M}\|_\beta^{2-k}  \big)
			\|({\bar Z}_{\lfloor r\rfloor_\tau}^{N})^{:k:}\|_{-\alpha}\\
		&+(t-r)^{-\frac {\kappa+\alpha+\beta}2} \sum_{k=1}^3\big(1+\|Y_{r}^{N,M}\|^{3-k}_\beta\big)
		\|({\bar Z}_{r}^{N})^{:k:}-({\bar Z}_{\lfloor r\rfloor_\tau}^{N})^{:k:} \|_{-\alpha} \Big\}dr,
	\end{aligned}
\end{equation}
where on the right-hand side of \eqref{L1EST},  by the identity \eqref{Tmild-2-NM} for any  $r \in[0,t]$  we have the term
\begin{equation}\label{Tmild-2-yr-r}\begin{aligned}
		Y_r^{N,M}-Y_{\lfloor r \rfloor_\tau }^{N,M}=&
		\int_\tau^{\lfloor r \rfloor_\tau\vee\tau}	(S_{r-\lfloor r\rfloor_\tau}-I)\frac{P_N S_{\lfloor r \rfloor_\tau-u} \Psi(Y_{\lfloor u\rfloor_\tau}^{N,M},\underline{\bar Z}_{\lfloor u\rfloor_\tau}^{N} )}
		{1+\tau \|\Psi(Y_{\lfloor u\rfloor_\tau}^{N,M},\underline{\bar Z}_{\lfloor u\rfloor_\tau}^{N} )\| _{-\alpha}}  du
	\\	&+\int^{r\vee\tau}_{\lfloor r \rfloor_\tau\vee\tau}\frac{P_N S_{r-u} \Psi(Y_{\lfloor u\rfloor_\tau}^{N,M},\underline{\bar Z}_{\lfloor u\rfloor_\tau}^{N} )}
		{1+\tau \|\Psi(Y_{\lfloor u\rfloor_\tau}^{N,M},\underline{\bar Z}_{\lfloor u\rfloor_\tau}^{N} )\|  _{-\alpha}}  du.
	\end{aligned}
\end{equation}
By construction of $Y^{N,M}$ we only need to estimate \eqref{Tmild-2-yr-r} for $r\in[\tau,t]$ with $t\in[\tau,t_m\wedge\sigma^N_R)$.
Using Lemma \ref{Heat-Smooth1} and \eqref{est-PN}  we estimate  the terms on the right-hand side of \eqref{Tmild-2-yr-r} and obtain that for any positive $\lambda<2$
 \begin{equation}\label{Sr-I}
  \begin{aligned}
   &(\lfloor r \rfloor_\tau-u)^{\frac{\lambda+\kappa}{2}}
   \|(S_{r-\lfloor r\rfloor_\tau}-I){P_N S_{\lfloor r \rfloor_\tau-u} \Psi(Y_{\lfloor u\rfloor_\tau}^{N,M},\underline{\bar Z}_{\lfloor u\rfloor_\tau}^{N} )}\|_\beta \\
    \lesssim &\tau^{\frac\lambda 2}
    \big(\|F(Y_{\lfloor u\rfloor_\tau}^{N,M})\|_\beta+(\lfloor r \rfloor_\tau-u)^{-\frac{\alpha+\beta}{2}}
   \|\widetilde{\Psi}(Y_{\lfloor u\rfloor_\tau}^{N,M},\underline{\bar Z}_{\lfloor u\rfloor_\tau}^{N} )\|_{-\alpha}\big),\\
  & (r-u)^{\frac{\kappa}{2}}
   \|P_N S_{r-u} \Psi(Y_{\lfloor u\rfloor_\tau}^{N,M},
    \underline{\bar Z}_{\lfloor u\rfloor_\tau}^{N} )\|_{\beta}\\
    \lesssim &\|F(Y_{\lfloor u\rfloor_\tau}^{N,M})\|_\beta+(r-u)^{-\frac{\alpha+\beta}{2}}
   \|\widetilde{\Psi}(Y_{\lfloor u\rfloor_\tau}^{N,M},\underline{\bar Z}_{\lfloor u\rfloor_\tau}^{N} )\|_{-\alpha}.
  \end{aligned}
\end{equation}
Here we choose $\lambda\in(10\mu,~2-5\alpha-\beta)$ and recall $K=\tau^{-\mu}$ with $\mu\in(0,\frac15-\frac{5\alpha+\beta}{10})$. Then inserting \eqref{Sr-I} into  \eqref{Tmild-2-yr-r},
by \eqref{F-PSI} and the inequality   $x^{-a}\lesssim x^{-b}$ for uniform $x\in[0,T]$ with $0\leq a\leq b$
\begin{equation} \label{Yt-tmild-ne}
 \begin{aligned}
 	&1_{\Omega_{N,m-1}^K} \|Y_r^{N,M}-Y_{\lfloor r \rfloor_\tau }^{N,M}\|_{\beta}   \lesssim_R 	\tau^{-3\mu} \int^{r\vee\tau}_{\lfloor r\rfloor_\tau\vee\tau}
 	(r-u)^{-\frac{\kappa+\alpha+\beta}{2}}
 	  {\lfloor u\rfloor}^{-2 \alpha -\kappa}_\tau du\\
&+ 	\tau^{\frac\lambda 2-3\mu}  \int_\tau^{\lfloor r \rfloor_\tau\vee\tau}
 	(\lfloor r \rfloor_\tau-u)^{-\frac{\kappa+\lambda+\alpha+\beta}{2}}
 	{\lfloor u\rfloor}^{-2 \alpha -\kappa}_\tau du \\ &
 		 \lesssim_R
		\tau^{ \frac\lambda2-3\mu}+
	  \tau^{1-\frac{5\alpha+\beta}{2}-3\mu-2\kappa} \lesssim_R\tau^{ \frac\lambda2-3\mu},
 \end{aligned}
\end{equation}
where for $\kappa>0$ sufficiently small, the second inequality follows from   \eqref{Tineq2} and  \eqref{coef}, the third inequality follows from the inequality $\tau^a\lesssim \tau^b$   with $0\le b\le a$.  Immediately, we have
\begin{equation}
\label{Yt-tmild-ne2}
1_{\Omega_{N,m-1}^K} \|Y_r^{N,M}\|_\beta \lesssim 1+K,~~r\in[0,t_m\wedge \sigma^N_R).
\end{equation}
Inserting  \eqref{Yt-tmild-ne} and \eqref{Yt-tmild-ne2}  into \eqref{L1EST}, again by \eqref{coef} and \eqref{Tineq2} we have
for   $\delta_n\in(0,\frac \alpha{2n})$, $n=1,2,3$
\begin{equation}\label{L1EST-02}
	\begin{aligned}
	1&_{\Omega_{N,m-1}^K}	\|L_1\|_{\beta}   \lesssim_R
		\int_\tau^{t\vee\tau}  \Big\{\tau^{\frac\lambda2- 5\mu}(t-r)^{-\frac \kappa2}   +(t-r)^{-\frac {\kappa+\alpha+\beta}2}
		\big(\tau^{ \frac\lambda2-4\mu}
			{\lfloor r\rfloor}_\tau^{-\alpha-\kappa}
	\\&~	+ \tau^{\frac\alpha{2}-\delta_1-2\mu} {\lfloor r\rfloor}_\tau^{\delta_1-\frac \alpha2-\kappa}
	 	+   \tau^{\frac\alpha{4}-\delta_2- \mu} {\lfloor r\rfloor}_\tau^{\delta_2-\frac 54\alpha-\kappa}
	+\tau^{\frac\alpha{6}-\delta_3} {\lfloor r\rfloor}_\tau^{\delta_3 -\frac5{3}\alpha-\kappa}      \big)
	  \Big\}dr\\
	&~  \lesssim_R   \tau^{\min\{ \frac\lambda2-5\mu, \frac\alpha 2-\delta_1-2\mu, \frac\alpha 4-\delta_2- \mu, \frac\alpha 6-\delta_3     \}}
	\end{aligned}
\end{equation}
uniformly for $t\in[0,t_m\wedge \sigma^N_R)$.\\
{\textbf{Item $L_2$.~}} Let $t\in[0,t_m\wedge \sigma^N_R)$, we note that
\[
\|L_2\|_\beta  \lesssim \tau
\int^{t\vee\tau}_\tau  \|P_N S_{t-r} \Psi(Y_{\lfloor r\rfloor_\tau}^{N,M},\underline{\bar Z}_{\lfloor r\rfloor_\tau}^{N} )\|_\beta \|\Psi(Y_{\lfloor r\rfloor_\tau}^{N,M},\underline{\bar Z}_{\lfloor r\rfloor_\tau}^{N} )\|_{-\alpha} dr,
\]
where by Lemma  \ref{Multi-ineq} and Young's inequality
\[
\|\Psi(Y_{\lfloor r\rfloor_\tau}^{N,M},\underline{\bar Z}_{\lfloor r\rfloor_\tau}^{N} )\|_{-\alpha}
\lesssim  \sum_{k=0}^3 \big(1+\|Y_{\lfloor r\rfloor_\tau}^{N,M}\|_\beta^{3-k} \big)  \|({\bar Z}_{\lfloor r\rfloor_\tau}^{N})^{:k:}\|_{-\alpha} .
\]
The following procedure is similar as we  treat   $L_1$. For     $\kappa>0$ sufficiently small
\begin{equation}\label{L2-est1}
	\begin{aligned}
		 	1_{\Omega_{N,m-1}^K}	\|L_2\|_\beta
		\lesssim _R	&  \tau   \int^{t\vee\tau}_\tau\big\{ \tau^{-3\mu}(t-r)^{-\frac{\kappa}2}
			  +	\tau^{-2\mu}	(t-r)^{-\frac{\alpha+\beta+\kappa}2} {\lfloor r\rfloor}_\tau^{ -\kappa}
			  \\& +	
			    	(t-r)^{-\frac{\alpha+\beta+\kappa}2} {\lfloor r\rfloor}_\tau^{-2\alpha-\kappa}  	\big\}
	\cdot\big	\{ \tau^{-3 \mu} +{\lfloor r\rfloor}_\tau^{-2\alpha-\kappa}\big\}	  	
   	dr\\
		\lesssim_R& \tau^{1-6\mu}+\tau^{1-2\alpha-\kappa}.
	\end{aligned}
\end{equation}
 Hence, we conclude from \eqref{L1EST-02} and \eqref{L2-est1}   that there exists $\nu>0$ such that  for uniform  $N,M\in\mathbb N$ and   $m=0,1,...,M$
\begin{equation}\label{V=L12-est1}
\sup_{t\in[0,t_m\wedge\sigma^N_R]}	1_{\Omega_{N,m-1}^K}\|V_t^{N,M}\|_\beta \lesssim_R \tau^{\nu}  .
\end{equation}
{\textbf {Step 2: Estimate of $\hat{Y}_t^{N,M}$.}}  \eqref{V=L12-est1}, \eqref{Test-ZN-HOLD-INI2}, together with  \cite[Theorem 5.1]{LR13} imply that \begin{equation}\label{Test-Yset-ti-EHOLD1}
	\sup_{N,M\in\mathbb N}\E \sup_{m\in\{0,...,M\}}\sup_{t\in[0,t_m\wedge\sigma^N_R]}1 _{\Omega^K_{N,m-1}} \|{\hat Y}_{t }^{N,M}\|^p_{\beta}<\infty.
\end{equation}
\medskip

In addition, Theorem \ref{Z-INI} implies that for uniform $N\in\mathbb N$
 \begin{equation} \label{t=T,1}
	\lim_{R\rightarrow\infty}P(\sigma^N_{R}=T)=1.
\end{equation}
 Finally, \eqref{Test-Yset-ti-EHOLD} follows from \eqref{V=L12-est1}, \eqref{Test-Yset-ti-EHOLD1} and \eqref{t=T,1}.
\end{proof}

By definition \eqref{set-bdd} we easily have $\Omega_{N,m}^K\subset \Omega_{N,{m-1}}^K$ for any $K>0$ and $m=0,1...,M$. As an immediate consequence of \eqref{Test-Yset-ti-EHOLD}, we have
\begin{corollary} \label{Test-Yset-Ti-HOLD-Th2}
Assume the setting in  Lemma \ref{Test-Yset-Ti-HOLD-Th}.
Then
\begin{equation}\label{Test-Yset-ti-EHOLD21}
	\sup_{N,M\in\mathbb N}\E \sup_{m\in\{0,...,M\}}\sup_{t\in[0,t_m]}1 _{\Omega^K_{{N,m}}}\|Y_t^{N,M}\|^p_{\beta}<\infty.
\end{equation}
\end{corollary}
\medskip
\begin{proof}[Proof of Theorem \ref{Test-Y-HOLD-Th}]
Let $R>0$ and $\sigma^N_R$ is given by \eqref{stp01}. With the help of \eqref{Test-Yset-ti-EHOLD21} and \eqref{t=T,1}, it is sufficient to prove for the above $K=\tau^{-\mu}$
\begin{equation}\label{Test-Yset-ti-EHOLD2}
	\sup_{N,M\in\mathbb N}\E \sup_{m\in\{0,...,M\}}\sup_{t\in[0,t_m\wedge\sigma^N_R]}1 _{\bar\Omega^K_{{N,m}}} \|{ Y}_{t }^{N,M}\|^p_{\beta}<\infty.
\end{equation}
By \eqref{Tmild-2-NM} it is easy to have a rough estimate that for uniform $t\in[0,T]$
\begin{equation}\label{tauYN-1}
\|Y^{N,M}_t\|_\beta\lesssim\tau^{-1}\int^{t}_0(t-s)^{-\frac{\alpha+\beta+\kappa}{2}}
ds \lesssim \tau^{-1}
\end{equation}
with $\kappa>0$ sufficiently small.
We can also see that
\begin{equation*}\label{set-depOmeg}
\bar\Omega^K_{{N,m}}=\bar\Omega^K_{{N,m-1}}+
\Omega^K_{{N,m-1}}\cap\{ \|Y_{t_m}^{N,M}\|_{\beta} >K \},
\end{equation*}
which implies that
\begin{equation}\label{set-depOmeg}
1_{\bar\Omega^K_{{N,m}}}=1_{\bar\Omega^K_{{N,m-1}}}+
1_{\Omega^K_{{N,m-1}}}\cdot1_{\{ \|Y_{t_m}^{N,M}\|_{\beta} >K \}}=\sum_{i=0}^m 1_{ \Omega^K_{{N,i-1}}}\cdot
1_{\{ \|Y_{t_i}^{N,M}\|_{\beta} >K \}}
\end{equation}
with $1_{\bar\Omega^K_{N,-1}}=0$. Inserting \eqref{Test-Yset-ti-EHOLD} and \eqref{tauYN-1} into  \eqref{set-depOmeg} and by Chebyshev's inequality, we have
\begin{equation*}\begin{aligned}
& ~~~~~~~~\E\sup_{m\in\{0,...,M\}}\sup_{t\in[0,t_m\wedge\sigma^N_R]} 1_{\bar\Omega^K_{N,m}} \|Y_{t }^{N,M}\|_\beta^p\\
\lesssim &
\tau^{-p}K^{-\frac{p+1}\mu}\sum_{i=0}^M \E(
 { 1_{\Omega^K_{N,i-1}} \|Y_{t_i}^{N,M}\|^{\frac{p+1}{\mu}}_{\beta}  })\lesssim \tau(M+1)\lesssim 1.
 \end{aligned}
\end{equation*}
 \eqref{Test-Yset-ti-EHOLD2} holds and the proof is completed.
\end{proof}
In the end of this subsection, recall $X$ the solution  to \eqref{initial-Eq} and its
  full-discrete approximation $X^{N,M}$ given in \eqref{nNMX}.   We let $\beta$  close to $\alpha$ ($\beta>\alpha$) and  conclude by \eqref{est-Y-EHOLD}, \eqref{Test-Y-EHOLD} and \eqref{Test-ZN-HOLD-INI2}
   that   for every $\alpha\in(0,1/3)$, $\alpha'>0$ and $p\ge2$
\begin{equation} \label{TX=Y+Z-EST}
\E \sup_{0\leq t\leq T}\|X_{t}\|^p_{-\alpha}<\infty,~~
\sup_{N,M\in\mathbb N}\E \sup_{0\leq t\leq T} t^{\alpha'p}\|X_{t}^{N,M}\|^p_{-\alpha} <\infty.
\end{equation}
\medskip
\begin{lemma} \label{lemTineq}
Let $a,b>0, c\ge0$ with $a+b<1+c$. Then\\
 $(i)$ for any $t>0$
 \begin{equation}\label{ineq}
	t^c\int_0^t (t-s)^{-a}s^{-b}ds \lesssim  t^{1+c-a-b}.
\end{equation}
$(ii)$ for any $t\geq \tau$
\begin{equation}  \label{Tineq2}
t^c\int_\tau^t (t- s )^{-a}  {\lfloor s\rfloor}^{-b}  ds \lesssim (t-\tau)^{1+c-a-b}.
\end{equation}
\end{lemma}

\subsection {Strong convergence rates}\label{SEC4-2} In this subsection, we analyze the error estimate of $Y$ and $Y^{N,M}$ in the space $\C^\beta$, utilizing two distinct norms $\|\cdot\|_\beta$ and $t^\gamma\|\cdot\|_\beta$ for some $\gamma>0$ (see Theorems \ref{MAIN-THM1} and \ref{MAIN-THM1-2} respectively). The parameter $\gamma$ is crucial for properly defining the linear term $\bar Z_t$ as $t$ approaches $0$, and the latter norm helps achieve superior convergence rates. Consequently, leveraging the convergence rate for the Galerkin approximation $\bar Z$ of $\bar Z$ obtained in \cite{MZ2019} (see \eqref{ZTN-ZN-n-ini} for details), we establish space and time convergence rates of the approximate scheme \eqref{nNMX}.

Let $\alpha\in(0,1)$, $X_0\in\C^{-\alpha}$, $p\ge2$.  \cite[Theorem3.5]{MZ2019}  showed that for any  $\kappa,\kappa_1,\delta>0$
\begin{equation}\label{ZTN-ZN-n-ini}
      \begin{aligned}
    &\mathbb E\sup_{0\leq t\leq T}  t^{ (n-1)(\alpha+\kappa)p +\kappa_1  p}
    \|\bar Z^{:n:}_{t}-(\bar Z^N_{t})^{:n:}\|_{ {-\alpha}}^p\\
    & \lesssim   {(\log N)^{2p}}{N^{-2\kappa_1 p}}+ (1+N^2)^{-\frac{ p(\alpha-\delta)} {2}}.
\end{aligned}
\end{equation}

 Following the notations in the above section, for some fixed $R>0$ sufficiently large, we define stopping times
 \begin{equation*}\label{stp1}
 \sigma^R:=\inf \Big\{t\leq T:  \max \big(  \|Y_t\|_\beta, t^ {(n-1)(\alpha+\kappa)} \|\bar Z_t^{:n:}\|_{ {-\alpha}},n=1,2,3 \big) \geq R \Big\}\wedge T,
\end{equation*}
\[\sigma_{N,M}^{R} :=\inf\Big\{t\leq T:   \|Y^{N,M}_t \|_{\beta}  >R  \Big\}\wedge T,\]
\[ \nu_{N}^{R}:=\inf \Big\{t\leq T:
               \max_{n=1,2,3}
  t^{(n-1)(\alpha+\kappa)+{\kappa}_1 }\|
 (\bar Z^N_t)^{:n:}\|_{ {-\alpha}}   >R  \Big\}\wedge T,
\]
with $\kappa>0$ sufficiently small and  let
 \begin{equation}\label{stp2}
 \varsigma_{N,M}^{R }:=\sigma^R\wedge \sigma_{N,M}^{R} \wedge \nu_{N}^{R}.
\end{equation}
Then by \eqref{Test-ZN-HOLD-INI2}, \eqref{est-Y-EHOLD} and \eqref{Test-Y-EHOLD} it is obvious that for uniform $N,M\in\mathbb N$
 \begin{equation}\label{stp2-lit}
 \lim_{R\to\infty}\varsigma_{N,M}^{R }=T.
\end{equation}

In the following theorem we consider pathwise error estimate for space-time approximation $Y^{N,M}$ given by \eqref{Tmild-2-NM}.
\begin{theorem}\label{MAIN-THM1}
Let $p\ge2$, $\alpha, \beta $ satisfy \eqref{coef} and
$\delta >0$.  Then for  uniform large $N,M\in \mathbb N$
\begin{equation} \label{Y-YNM-new}
 \begin{aligned}
 \Big(\E \sup_{t\in[0, T]}
 \|Y_t-Y_{t}^{N,M}\|_\beta^p\Big)^{1/p} \lesssim &   N^{\delta-\min\{2-5\alpha-\beta,\alpha\}}+
 M^{\delta-\min\{1-\frac{5\alpha+\beta}2,\frac{\alpha}{6}\}}. \end{aligned}
\end{equation}
\end{theorem}
\begin{proof} Since by Jensen's inequality
 \begin{equation*}
 \begin{aligned}
 &~~~~\E \sup_{t\in[0, T]}
 \|Y_t-Y_{t}^{N,M}\|_\beta^p\le
 \E \sup_{t\in[0, \varsigma_{N,M}^{R }]}
 \|Y_t-Y_{t}^{N,M}\|_\beta^p
 \\
 &+\big(\mathbb P(\varsigma_{N,M}^{R }<T)\big)^{1/2}\cdot\big(\E \sup_{t\in[0, T]}
 \|Y_t\|^{2p}_\beta+\E \sup_{t\in[0, T]}\|Y_{t}^{N,M}\|_\beta^{2p}\big)^{1/2}
  \end{aligned}
\end{equation*}
 with $\varsigma_{N,M}^{R }$ given by \eqref{stp2} and $\kappa_1=\tilde \kappa\in(0,1-\frac{5\alpha+\beta}{2})$. Then together with \eqref{stp2-lit}, \eqref{est-Y-EHOLD} and \eqref{Test-Y-EHOLD}, it is sufficient  to prove for uniform large $N,M\in\mathbb N$
\begin{equation}  \label{YY-YNM}
 \begin{aligned}
 \Big(\E \sup_{t\in[0, \varsigma_{N,M}^{R}]}
 \|Y_t-Y_{t}^{N,M}\|_\beta^p\Big)^{1/p} \lesssim_R &   N^{\delta-\min\{2-5\alpha-\beta,\alpha\}}+
 M^{\delta-\min\{1-\frac{5\alpha+\beta}2,\frac{\alpha}{6}\}}
 \end{aligned}
\end{equation}
with fixed large $R$. Comparing with \eqref{Tmild-2} and \eqref{Tmild-2-NM} we have decomposition
\begin{equation}\label{decomp-2}
    Y_t- Y_t^{N,M}=I_1+I_2+L_1+L_2
\end{equation}
  with   $L_1,L_2$ given in \eqref{L1L2} and
  \begin{equation}\begin{aligned}\label{I1I2}
 I_1 =&\int_0^t \big(I-P_N\big )S_{t-s}\Psi(Y_s ,\underline{\bar Z}_s )ds,\\
 I_2= &\int_0^t P_NS_{t-s}\Big\{ \Psi(Y_s ,\underline{\bar Z}_s )
     -    \Psi(Y_s^{N,M},\underline{\bar Z}_s^{N}) \Big\} ds.
  \end{aligned}\end{equation}
{\textbf{Estimate for  $I_1$:}} According to the condition \eqref{coef}, we can choose  $\delta\in(0,2-5\alpha-\beta)$  and set $\lambda=2-5\alpha-\beta-\delta$ ($\lambda>0$).  Then for uniform $t\in[0, \varsigma_{N,M}^{R}]$
   \begin{equation} \label{Test-I11}\begin{aligned}
   		 \|I_1\|_\beta
   		& \lesssim  \frac{(\log N)^2}{N^\lambda}\int_0^t
   	\|	S_{t-s}\Psi(Y_s ,\underline{\bar Z}_s ) \|_{\beta+\lambda} ds\\
   &	\lesssim_R  \frac{(\log N)^2}{N^\lambda}\int_0^t
    (t-s)^{-\frac{\alpha+\beta+\lambda}{2}} s^{- 2\alpha-\kappa} ds\\
   	&\lesssim_R  \frac{(\log N)^2}{N^{2-5\alpha-\beta-\delta}}.
   		 \end{aligned}\end{equation}
The first inequality follows from \eqref{est-P-PN}, the second inequality follows from  \eqref{Heat-Smooth} and \eqref{F-PSI}, and the last inequality follows from  \eqref{ineq}. \\  	
{\textbf{Estimate for  $I_2$:}}  For uniform $t\in[0, \varsigma_{N,M}^{R}]$ we have
  \begin{equation}\label{Test-I1}\begin{aligned}
            \|I_2\|_\beta
    \lesssim_R &\int_0^t
     (t-s)^{-\frac{\alpha+\beta+\kappa}{2}}  \Big \{s^{-\alpha- \kappa }
      \|Y_s-Y_s^{N,M}\|_{{\beta}}
    +  \sum_{n=1}^3  \|(\bar Z_s)^{:n:} -(\bar Z^{N}_s)^{:n:}\|_{{-\alpha}}  \Big\}ds
\\ \lesssim_R  &\int_0^t   (t-s)^{-\frac{\alpha+\beta+\kappa}{2}} s^{-\alpha-\kappa }  \|Y_s-Y_s^{N,M}\|_{{\beta}}    ds
    \\&~~~~~~ +   \sum_{n=1}^3\sup_{0\leq s \leq T} s^{(n-1)(\alpha+\kappa)+\tilde\kappa}\|(\bar Z_s)^{:n:} -(\bar Z^{N}_s)^{:n:}\|_{{-\alpha}},
 \end{aligned}\end{equation}
where the first inequality follows by \eqref{Heat-Smooth}, \eqref{est-PN} and \eqref{F-PSI2},  the second inequality follows by \eqref{coef}, \eqref{ineq} and $\tilde\kappa<1-\frac{5\alpha+\beta}{2}$.

 To obtain better convergence rate we restart the  estimates for $L_1$ and $L_2$.\\
{\textbf{Estimate for  $L_1$:  }} Recall \eqref{L1EST} and we have for uniform $t\in[0, \varsigma_{N,M}^{R}]$

\begin{equation}\label{L1ESTRE}
	\begin{aligned}
&	\|L_1\|_{\beta} \lesssim_R	\int^{t\vee\tau}_\tau  \Big\{ (t-r)^{-\frac {\kappa+\alpha+\beta}2} \lfloor r\rfloor_\tau^{-\alpha-\kappa}
		\|Y_{r}^{N,M}-Y_{\lfloor r\rfloor_\tau}^{N,M}\|_\beta
		 \\
		&+(t-r)^{-\frac {\kappa+\alpha+\beta}2} \sum_{n=1}^3 \|({\bar Z}_{r}^{N})^{:n:}-({\bar Z}_{\lfloor r\rfloor_\tau}^{N})^{:n:} \|_{-\alpha} \Big\}dr,
	\end{aligned}
\end{equation}
where, by  the definition \eqref{stp2} and inserting \eqref{F-PSI} and \eqref{Sr-I} into \eqref{Tmild-2-yr-r}, we have for any positive $\lambda\le2$
 \begin{equation} \label{Yt-tmild-neRE}
 \begin{aligned}
 		& \|Y_r^{N,M}-Y_{\lfloor r \rfloor_\tau }^{N,M}\|_{\beta}   \lesssim_R 	
 	\tau^{\frac\lambda 2}  \int_{\tau\vee\tau}^{\lfloor r \rfloor_\tau\vee\tau}
 	(\lfloor r \rfloor_\tau-u)^{-\frac{\kappa+\lambda+\alpha+\beta}{2}}
 	{\lfloor u\rfloor}^{-2 \alpha -\kappa}_\tau du \\ &
 		+ \int^{r\vee\tau}_{\lfloor r\rfloor_\tau\vee\tau}
 	(r-u)^{-\frac{\kappa+\alpha+\beta}{2}}
 	  {\lfloor u\rfloor}^{-2 \alpha -\kappa}_\tau du
 	\end{aligned}
\end{equation}
 with $\kappa>0$ arbitrarily small. Set  the above $\lambda=2-{5\alpha-\beta}-4\kappa$. Using \eqref{Tineq2} in \eqref{Yt-tmild-neRE} we have
 \[
\|Y_r^{N,M}-Y_{\lfloor r \rfloor_\tau }^{N,M}\|_{\beta}
 	  \lesssim_R
	  \tau^{1-\frac{5\alpha+\beta}{2}-2\kappa}.\]
Then inserting into \eqref{L1ESTRE} and by \eqref{coef} and \eqref{Tineq2},  we have
 \begin{equation}\label{L1ESTReE}
	\|L_1\|_{\beta} \lesssim_R\tau^{ 1-\frac{5\alpha+\beta}2-\kappa}+\sum_{n=1}^3 \sup_{0\leq r \leq T}{\lfloor r \rfloor }_\tau ^{\mu_n}\|({\bar Z}_{r}^{N})^{:n:}-({\bar Z}_{\lfloor r\rfloor_\tau}^{N})^{:n:} \|_{-\alpha},
\end{equation}
with $\mu_1=\frac{\alpha}{2}-\delta_1+\kappa$ and $\mu_n=\frac{n^2+1}{2n}\alpha-\delta_n+\kappa$ for $n=2,3$.\\
{\textbf{Estimate for $L_2$: }} Similarly as to obtain \eqref{L2-est1}, by  the definition \eqref{stp2} we obtain
 for uniform $t\in[0, \varsigma_{N,M}^{R}]$
 with  $\kappa>0$ sufficiently small
\begin{equation}\label{L2-est1re}
			\|L_2\|_\beta
		\lesssim _R	  \tau   \int^{t\vee\tau}_\tau
			    	(t-r)^{-\frac{\alpha+\beta+\kappa}2}
    {\lfloor r\rfloor}_\tau^{-4\alpha-2\kappa}  	dr
		\lesssim_R \tau^{1-2\alpha-\kappa},
\end{equation}
where the last inequality follows from \eqref{coef} and \eqref{Tineq2}.

Putting together the above estimates for $I_1$, $I_2$, $L_1$ and $L_2$, we conclude for uniform $t\in[0, \varsigma_{N,M}^{R}]$
 with $\kappa>0$ sufficiently small, $\delta>0$, $\delta_n\in(0, \frac{\alpha}{2n})$ and $\tilde{\kappa}\in(0,1-\frac{5\alpha+\beta}{2})$
  \begin{equation*} \label{Y-YNM}
 \begin{aligned}
 \|Y_t-Y_{t}^{N,M}& \|_\beta \lesssim_R  \int_0^t   (t-s)^{-\frac{\alpha+\beta+\kappa}{2}} s^{-\alpha-\kappa }  \|Y_s-Y_s^{N,M}\|_{{\beta}}    ds +\frac{(\log N)^2}{N^{2-5\alpha-\beta-\delta}}
    \\&  +   \sum_{n=1}^3\sup_{0\leq s \leq T} s^{(n-1)(\alpha+\kappa)+\tilde\kappa}\|(\bar Z_s)^{:n:} -(\bar Z^{N}_s)^{:n:}\|_{{-\alpha}}\\&
   +\sum_{n=1}^3 \sup_{0\leq r \leq T} {\lfloor r \rfloor }_\tau^{\mu_n}\|({\bar Z}_{r}^{N})^{:n:}-({\bar Z}_{\lfloor r\rfloor_\tau}^{N})^{:n:} \|_{-\alpha}+
   \tau^{ 1-\frac{5\alpha+\beta}2-\kappa}.
 \end{aligned}
\end{equation*}
Taking expectation and using Gronwall's inequality, together with \eqref{coef} and \eqref{Tineq2}, \eqref{ZTN-ZN-n-ini} and Theorem \ref{Z-INI}  we have that \eqref{YY-YNM} follows for any $p\ge2$.
\end{proof}

Below we reconsider the error estimate between $Y^{N,M}$ and $Y$ in the same Besov space $\C^\beta$ with the  norm replaced by $t^{\gamma}\|\cdot\|_\beta$ for some $\gamma>0$. There exists some change for the values of  $\alpha,\beta$. Instead of the condition \eqref{coef}, we assume that
 the  positive constants $\alpha<\beta,\gamma$ satisfy
\begin{equation}\label{coef3}
3\alpha<1,~\frac{3\alpha+\beta}{2}<1,~
\max\big\{\frac{\alpha+\beta}{2},2\alpha\big\}<\gamma<1-\alpha.
\end{equation}

\begin{theorem}\label{MAIN-THM1-2}
Let $p\ge2$, $\alpha, \beta,\gamma $ satisfy \eqref{coef3} and
$\delta >0$.  Then for  uniform large $N,M\in \mathbb N$
\begin{equation} \label{Y-YNM-new-2}
 \begin{aligned}
 \Big(\E \sup_{t\in[0, T]}
 t^{\gamma p}\|Y_t-Y_{t}^{N,M}\|_\beta^p \Big)^{1/p}\lesssim &   N^{\delta-\alpha}+
 M^{\delta-{\alpha}/{6}}.
 \end{aligned}
\end{equation}
\end{theorem}
\begin{proof}  As discussed in Theorem \ref{MAIN-THM1}, it is sufficient  to consider \eqref{Y-YNM-new-2} for $t\in[0, \varsigma_{N,M}^{R}]$ with fixed large $R$, $\varsigma_{N,M}^{R}$ is defined in \eqref{stp2} by  letting $\kappa_1=\bar\kappa\in(0,1-\frac{\alpha+\beta}{2})$.
Following  the decomposition \eqref{decomp-2}
we estimate the terms $I_1,I_2,L_1$ and $L_2$ again in the space $\mathcal C^{\beta}$ with the new norm $t^\gamma \|\cdot\|_{\beta}$.\\
{\textbf{Term  $I_1$:}} According to the condition \eqref{coef3}, for any $\delta\in(0,2-\alpha-\beta)$  we set $\lambda=2-\alpha-\beta-\delta$ ($\lambda>0$). Multiplying by $t^\gamma$ on both sides of \eqref{Test-I11} we have for uniform $t\in[0, \varsigma_{N,M}^{R}]$
   \begin{equation} \label{Test-I11-2}\begin{aligned}
   		 t^{\gamma}\|I_1\|_\beta
   		 \lesssim_R   \frac{(\log N)^2}{N^\lambda} t^{\gamma} \int_0^t
    (t-s)^{-\frac{\alpha+\beta+\lambda}{2}} s^{- 2\alpha-\kappa} ds
   	\lesssim_R  \frac{(\log N)^2}{N^{2-\alpha-\beta-\delta}},
   \end{aligned}\end{equation}
where the second inequality  follows by   \eqref{ineq} with $c=\gamma$, $a=\frac{\alpha+\beta+\lambda}{2}$ and $b=2\alpha-\kappa$, and \eqref{coef3} ensures that such $a,b,c$ satisfy the conditions in  \eqref{ineq}. \\  	
{\textbf{Term  $I_2$:}} Let positive $\bar{\kappa}<1-2\alpha$.
According to the condition \eqref{coef3}, for any $\delta\in(0,1-2\alpha)$  we set $\lambda=1-2\alpha-\delta$ ($\lambda>0$). Multiplying by $t^\gamma$ on both sides of \eqref{Test-I1} we have for uniform $t\in[0, \varsigma_{N,M}^{R}]$
  \begin{equation}\label{Test-I1-2}\begin{aligned}
           t^\gamma \|I_2\|_\beta
    \lesssim_R    &\int_0^t t^\gamma  (t-s)^{-\frac{\alpha+\beta+\kappa}{2}} s^{-\alpha-\kappa }  \|Y_s-Y_s^{N,M}\|_{{\beta}}    ds
    \\&+   \sum_{n=1}^3\sup_{0\leq s \leq T} s^{(n-1)(\alpha+\kappa)+\bar\kappa}\|(\bar Z_s)^{:n:} -(\bar Z^{N}_s)^{:n:}\|_{{-\alpha}},
 \end{aligned}\end{equation}
where the  inequality follows since by    \eqref{ineq}
\[
 t^\gamma \int_0^t (t-s)^{-\frac{\alpha+\beta+\kappa}{2}}
    s^{-2(\alpha+\kappa)-\bar\kappa}ds\lesssim t^{\nu_0}
\]
with $\nu_0:=1+\gamma-\frac{\alpha+\beta+\kappa}{2}-2(\alpha+\kappa)-\bar\kappa$, and  \eqref{coef3} and sufficiently small $\kappa>0$   ensures that $\nu_0$ is positive.\\
{\textbf{Term $L_1$:  }}
For any $\delta\in(0,1-\frac{\alpha+\beta}2)$,  we set $\lambda=2-\alpha-\beta-2\delta$ ($\lambda>0$) and $\kappa>0$ sufficiently small in \eqref{Yt-tmild-neRE}. By \eqref{Tineq2} for any $r\in[0,t]$ with  $t\in[0, \varsigma_{N,M}^{R}]$
 \begin{equation*} \label{Yt-tmild-neRE-2}
 \begin{aligned}
 	\lfloor r \rfloor_\tau ^\gamma	& \|Y_r^{N,M}-Y_{\lfloor r \rfloor_\tau }^{N,M}\|_{\beta}   \lesssim_R 	
 	\tau^{1-\frac {\alpha+\beta} 2-\delta}  .
 	\end{aligned}
\end{equation*}
Then inserting into \eqref{L1ESTRE}   we have for uniform $t\in[0, \varsigma_{N,M}^{R}]$
\begin{equation}\label{L1ESTRE-2}
	t^\gamma\|L_1\|_{\beta} \lesssim_R	\tau^{1-\frac {\alpha+\beta} 2-\delta}
	+
\sum_{n=1}^3 \sup_{0\leq r \leq T}{\lfloor r \rfloor }_\tau ^{\mu_n}\|({\bar Z}_{r}^{N})^{:n:}-({\bar Z}_{\lfloor r\rfloor_\tau}^{N})^{:n:} \|_{-\alpha},
\end{equation}
with $\mu_1=\frac{\alpha}{2}-\delta_1+\kappa$ and $\mu_n=\frac{n^2+1}{2n}\alpha-\delta_n+\kappa$ for $n=2,3$. The inequality follows since \eqref{Tineq2}, \eqref{coef3} and sufficiently small $\kappa>0$   ensures that there exists positive $\nu_1$  such that
\[
t^\gamma \int^{t\vee\tau}_\tau  (t-r)^{-\frac{\kappa+\alpha+\beta}2} \lfloor r\rfloor_\tau^{-\max\{\alpha+\kappa+\gamma,\mu_n,n=1,2,3\}} dr
			\lesssim t^{\nu_1}.
\]
{\textbf{Term $L_2$: }} Multiplying by $t^\gamma$ on both sides of \eqref{L2-est1re} we have for uniform $t\in[0, \varsigma_{N,M}^{R}]$
  with  $\kappa>0$ sufficiently small
\begin{equation}\label{L2-est1re-2}
			t^\gamma\|L_2\|_\beta
		\lesssim _R	  \tau^{1-\frac{\alpha+\beta}2}   t^\gamma \int^{t\vee\tau}_\tau
			    	(t-r)^{-\frac{\alpha+\beta+\kappa}2}
    {\lfloor r\rfloor}_\tau^{-\frac72\alpha+\beta-2\kappa}  	dr
		\lesssim_R \tau^{1-\frac{\alpha+\beta}2},
\end{equation}
where the last inequality follows from \eqref{coef3} and \eqref{Tineq2}.

Now combining with  the above results in \eqref{Test-I11-2}-\eqref{L2-est1re-2}, we conclude for uniform $t\in[0, \varsigma_{N,M}^{R}]$ with $\kappa>0$ sufficiently small, $\delta>0$, $\delta_n\in(0, \frac{\alpha}{2n})$ and $\bar{\kappa}\in(0,1-2\alpha)$
  \begin{equation*} \label{Y-YNM-22}
 \begin{aligned}
 t^\gamma\|Y_t-Y_{t}^{N,M}&\|_\beta \lesssim_R   t^\gamma\int_0^t   (t-s)^{-\frac{\alpha+\beta+\kappa}{2}} s^{-\alpha-\kappa }  \|Y_s-Y_s^{N,M}\|_{{\beta}}    ds +\frac{(\log N)^2}{N^{2-\alpha-\beta-\delta}}
    \\&  +   \sum_{n=1}^3\sup_{0\leq s \leq T} s^{(n-1)(\alpha+\kappa)+\bar\kappa}\|(\bar Z_s)^{:n:} -(\bar Z^{N}_s)^{:n:}\|_{{-\alpha}}\\&
   +\sum_{n=1}^3 \sup_{0\leq s \leq T} {\lfloor r \rfloor }_\tau^{\mu_n}\|({\bar Z}_{r}^{N})^{:n:}-({\bar Z}_{\lfloor r\rfloor_\tau}^{N})^{:n:} \|_{-\alpha}
   +\tau^{ 1-\frac{\alpha+\beta}2-\delta}.
 \end{aligned}
\end{equation*}
Taking expectation and using Gronwall's inequality, together with \eqref{coef3} and \eqref{Tineq2}, \eqref{ZTN-ZN-n-ini} and Theorem \ref{Z-INI}  we obtain that for any $p\ge2$
\[
 \Big(\E \sup_{t\in[0, T]}
 t^{\gamma p}\|Y_t-Y_{t}^{N,M}\|_\beta^p \Big)^{1/p}\lesssim    N^{\delta-\min\{2-\alpha-\beta,\alpha\}}+
 M^{\delta-\min\{1-\frac{\alpha+\beta}2,\frac{\alpha}{6}\}}.
\]
Then \eqref{Y-YNM-new-2} holds immediately by \eqref{coef3}.
\end{proof}

As a consequence of \eqref{nNMX}, \eqref{ZTN-ZN-n-ini} with $n=1$ and \eqref{Y-YNM-new-2} with $\beta$ close to  $\alpha$ ($\beta>\alpha$), together with the fact that $\|\cdot\|_{a}\lesssim \|\cdot\|_{a+\lambda}$ with any $\lambda>0$ and $a\in\mathbb R$, we   immediately obtain  the time and space convergence rates  for $X^{N,M}$ in \eqref{nNMX}, which is the main result throughout our paper.

\begin{theorem}\label{MAIN-THM1111}
Let $X_0\in\mathcal C^{-\alpha}$ with $\alpha\in (0,1/3)$, $\gamma>1-3\alpha$ and $p\ge2$. Then for  any  $\delta >0$
\begin{equation} \label{X-XNM}
 \Big(\E \sup_{t\in[0, T]}
 t^{\gamma p}\|X_t-X_{t}^{N,M}\|_{-\alpha}^p \Big)^{1/p} \lesssim    N^{\delta-\alpha}+
 M^{\delta-{\alpha}/{6}}
\end{equation}
 for uniform large $N,M\in \mathbb N$.
\end{theorem}


\appendix
\renewcommand{\appendixname}{Appendix~\Alph{section}}
\renewcommand{\theequation}{A.\arabic{equation}}
\section{Space-time white noise and Wiener chaos} \label{NOISE}
\begin{definition}
\label{WHITE-NOISE}
Let $\{\xi(\phi)\}_{\phi\in L^2(\mathbb R\times \mathbb T^d)}$ be a family of centered Gaussian random variables on a probability space $(\Omega,\mathcal F,\mathbb P)$ such that
\[
\mathbb E(\xi(\phi)\xi(\psi))=\langle\phi,\psi\rangle_{L^2(\mathbb R\times \mathbb T^d)},
\]
for all $\psi,\phi\in L^2(\mathbb R\times \mathbb T^d)$. Then $\xi$ is called a space-time white noise on $\mathbb R\times \mathbb T^d$. We interpret $\xi(\phi)$ as a stochastic integral and write
\[
\int_{\mathbb R\times \mathbb T^d}\psi(t,x)\xi(dt,dx):=\xi(\psi),~~\psi\in L^2(\mathbb R\times \mathbb T^d).
\]
\end{definition}
For any $n\in\mathbb N$, the multiple stochastic integrals (see \cite[Chapter 1]{Nu2006}) on $\mathbb R\times \mathbb T^d$ are defined for all symmetric functions $f$ in $L^2(\mathbb R\times \mathbb T^d)$, i.e. functions such that
\[
 f(z_1,\ldots,z_n)=f(z_{\sigma(1)},\ldots,z_{\sigma(n)}),~~z_i\in \mathbb R\times \mathbb T^d, j=1,2,\ldots,n,
\]
for any permutation $(\sigma(1),\ldots,\sigma(n))$ of $(1,\ldots,n)$. For such a symmetric function $f$ we denote its $n$-th interated stochastic integral by
\[
I_n(f):=\int_{(\mathbb R\times \mathbb T^d)^n}f(z_1,\ldots,z_n)\xi(\otimes_{i=1}^nds_i,\otimes_{i=1}^ndx_i),~z_i=(t_i,x_i)\in\mathbb R\times \mathbb T^d.
\]
\begin{theorem}{\cite[Theorem 1.1.2, Section 1.4]{Nu2006}}
\label{decom-wick}
Let $f$ be any symmetric function in $L^2((\mathbb R\times \mathbb T^d)^n)$. Then
\begin{equation}\label{iTO-EQ}
\mathbb E(I_n(f))^2=n!\|f\|_{L^2((\mathbb R\times \mathbb T^d)^n)}^2
\end{equation}
and
\begin{equation}\label{iTO-iEQ}
\mathbb E|I_n(f)|^p\leq (p-1)^{\frac{np}{2}}(\mathbb E|I_n(f)|^2)^{\frac{p}{2}}
\end{equation}
for every $p\geq 2$.
\end{theorem}
\bigskip
\renewcommand{\theequation}{B.\arabic{equation}}
\section{Functions with prescribed singularities} \label{kernal}
For  symmetric kernels $K_1,K_2\colon\mathbb Z^2\rightarrow(0,\infty)$, we denote its convolution
\[
  K_1\star K_2(m):=\sum_{l\in\mathbb Z^2} K_1(m-l)K_2(l)
\]
and for $N\in\mathbb N$ we set
\[
  K_1\star_{\leq N} K_2(m):=\sum_{|l|\leq N} K_1(m-l)K_2(l),~~
  K_1\star_{> N} K_2(m):=K_1\star K_2-K_1\star{_{\leq N}} K_2.
\]
For convolutions of the same kernel, we introduce
\begin{equation}\label{n}
   K\star^{1} K:=K,~ K\star^{n} K:= K\star({K\star^{n-1} K}),
\end{equation}
\begin{equation}\label{<Nn}
 K{\star^1}_{\leq N} K:= K, ~K{\star^n}_{\leq N} K:=K\star_{\leq N}({K{\star^{n-1}}_{\leq N} K}),
\end{equation}
where by simple calculation we actually obtain
 \begin{align*}
  K{\star^n}_{\leq N} K(m) =&\sum_{|l_i|\leq N,i=1,\ldots,n-1}  K(m-l_{n-1}) \prod^{n-1}_{i=1} K(l_i-l_{i-1}),
\end{align*}
with  the convention that $l_0 = 0$. Similarly for every $n\geq 2$, we denote
\begin{equation}\label{>Nn}
  K{\star^n}_{> N} K(m):=K\star^n K-K{\star^n}_{\leq N} K.
\end{equation}
Following the technique of \cite[Lemma 10.14]{Hai14}, we have the following estimates.
\begin{lemma}\label{KENEL-ES}{\cite[Corollary C.3]{TW20182}\cite[Lemma A.3]{MZ2019}}
Let $K^\gamma\colon\mathbb Z^2\rightarrow(0,\infty)$ be a symmetric kernel such that $K^\gamma(m)\lesssim \frac{1}{(1 + |m|^2)^{1-\gamma}}$, $\gamma\in[0,\frac{1}{n})$ for $n\in \mathbb N$.\\
{\rm(i)} If $\gamma>0$ then
\begin{equation*}\label{Kn}
\max\big\{K^\gamma \star^{n} K^\gamma (m),~~\sup_{N\geq 1}K^\gamma{\star{^{n}}}_{\leq N} K^\gamma (m)\big\}\lesssim \frac{1}{(1 + |m|^2)^{1-n\gamma}},
\end{equation*}
\begin{equation*}\label{K>Nn}
K^\gamma{\star{^{n}}}_{> N} K^\gamma (m) \lesssim  \left\{
\begin{aligned}
  &\frac{1}{(1 + |m|^2)^{1-n\gamma}},~~\text{if}~|m|\geq N,
\\&\frac{1}{(1 + |N|^2)^{1-n\gamma}},~~\text{if}~|m|< N.
\end{aligned}
\right.\end{equation*}
 {\rm(ii)} If  $\gamma=0$ then for every $\epsilon\in(0,1)$
\begin{equation*}\label{Kn_2}
\max\big\{ K^0\star^{n} K^0 (m),~~\sup_{N\geq 1}K^0{\star{^{n}}}_{\leq N} K^0 (m) \big\}
\lesssim   \frac{1}{(1 + |m|^2)^{1-\epsilon}},
\end{equation*}
\begin{equation*}\label{K>Nn_2}
K^0{\star{^{n}}}_{> N} K^0 (m) \lesssim \left\{
\begin{aligned}
  &\frac{1}{(1 + |m|^2)^{1-\epsilon}},~~\text{if}~|m|\geq N,
\\&\frac{1}{(1 + |N|^2)^{1-\epsilon}},~~\text{if}~|m|< N.
\end{aligned}
\right.\end{equation*}
\end{lemma}

\end{document}